\documentclass[review,onefignum,onetabnum]{siamart190516}

\usepackage{lipsum}
\usepackage{amsfonts}
\usepackage{graphicx}
\usepackage{epstopdf}
\ifpdf
  \DeclareGraphicsExtensions{.eps,.pdf,.png,.jpg}
\else
  \DeclareGraphicsExtensions{.eps}
\fi

\newsiamremark{remark}{Remark}
\newsiamremark{hypothesis}{Hypothesis}
\crefname{hypothesis}{Hypothesis}{Hypotheses}
\newsiamthm{claim}{Claim}

\headers{A Proximal-Gradient Algorithm for Crystal Surface Evolution}{K. Craig, J-G. Liu, J. Lu, J. L. Marzuola, and L. Wang}

 \title{A Proximal-Gradient algorithm for crystal surface evolution\thanks{Submitted   June 22, 2020.
\funding{KC was supported in part by the National Science Foundation (NSF) grant DMS-1811012 and a University of California Regent's Junior Faculty Fellowship.
JGL was supported in part by NSF DMS-1812573. JL was
supported in part by NSF under award DMS-1454939.  The research of the
The research of JLM was supported by NSF Grant DMS-1312874 and NSF
CAREER Grant DMS-1352353. 
LW was supported in part by NSF DMS-1903425 and DMS-1846854.
This collaboration is made possible thanks
to the NSF grant RNMS-1107444 (KI-Net).}}}

\author{Katy Craig\thanks{Dept of Math, University of California, Santa Barbara 
(\email{kcraig@math.ucsb.edu})}
\and Jian-Guo Liu\thanks{Dept of Math \& Dept of Physics, Duke University, Durham, NC 27708 (\email{jliu@math.duke.edu}}
\and Jianfeng Lu\thanks{Depts of Math, Chem, \& Phys, Duke University, Durham, NC 27708 (\email{jianfeng@math.duke.edu)}}  
\and
Jeremy L. Marzuola\thanks{Dept of Math, University of North Carolina, Chapel Hill, NC 27599 (\email{marzuola@math.unc.edu})}
\and
Li Wang\thanks{School of Math, University of Minnesota, Twin Cities (\email{wang8818@umn.edu})}}

\usepackage{amsopn}
\DeclareMathOperator{\diag}{diag}

\usepackage{amssymb,amsmath,graphicx,cite}
\usepackage[algo2e]{algorithm2e} 
\DeclareSymbolFont{bbold}{U}{bbold}{m}{n}
\DeclareSymbolFontAlphabet{\mathbbold}{bbold}

\newcommand{\R}{\mathbb{R}}

\newcommand{\Amat}{\mathsf{A}}
\newcommand{\Dmat}{\mathsf{D}}

\newcommand{\Lmat}{\mathsf{L}}
\newcommand{\Imat}{\mathsf{I}}

\newcommand{\TT}{\mathbb{T}}

\newcommand{\diver}{\text{div}\,}

\newcommand{\grad}{\nabla}

\DeclareMathOperator*{\argmax}{arg\,max}
\DeclareMathOperator*{\argmin}{arg\,min}

\DeclareMathOperator*{\sgn}{sgn}
\DeclareMathOperator*{\id}{id}
\newcommand{\la}{\langle}
\newcommand{\ra}{\rangle}
\newcommand{\supp}{\operatorname{supp}}
\newcommand{\E}{\mathcal{E}}

\def\dv{{\rm d}}

\begin{document}

\maketitle

\begin{abstract} As a counterpoint to   recent numerical methods for crystal surface evolution, which agree well with microscopic dynamics but suffer from significant stiffness that prevents simulation on fine spatial grids, we develop a new numerical method based on the macroscopic partial differential equation, leveraging its formal  structure as the gradient flow of the total variation energy, with respect to a weighted $H^{-1}$ norm. This   gradient flow structure relates to several metric space gradient flows of recent interest, including   2-Wasserstein     flows and their generalizations to nonlinear  mobilities. We develop a novel semi-implicit time discretization of the gradient flow, inspired by the classical minimizing movements scheme (known as the JKO scheme in the 2-Wasserstein case). We then use  a primal dual hybrid gradient (PDHG)     method to compute each element of the semi-implicit scheme. In one dimension, we prove convergence of the PDHG method to the semi-implicit scheme, under general integrability assumptions on the mobility and its reciprocal. Finally, by taking finite difference approximations of our PDHG method, we arrive at a fully discrete numerical algorithm, with iterations that converge at a rate independent of the spatial discretization: in particular, the convergence properties do not deteriorate as we refine our spatial grid. We close with several numerical examples illustrating the properties of our method, including facet formation at local maxima, pinning at local minima, and convergence as the spatial and temporal discretizations are refined.
\end{abstract}


% REQUIRED
\begin{AMS}  
 {35A15, 	47J25,	47J35 ,49J45, 49M29, 65K10, 82B21, 82B05}
\end{AMS}

\section{Introduction}

The evolution of a crystal surface near a fixed crystallographic plane of symmetry is determined by the desire to minimize the surface free energy \cite{GruberMullins67,Srolovitz94}. In terms of the height $h(x,t)$ of the surface, $x \in \Omega\subseteq \mathbb{R}^d$, $d \geq 1$, $t \geq 0$, the free energy is given by the well-known total variation energy
\begin{align} \label{firsttvenergy}
\E(h)= \int_\Omega  |\nabla h(x)| \,\dv x .
\end{align}
Facets on the crystal surface are identified with the regions $\{ x :\nabla h(x,t)=0 \}$.  
%(Sometimes an elastic dipole repulsive step-step interaction term is included in the energy, but we do not consider such terms in this work; see e.g. \cite{GruberMullins67,Kukta02-I}.)%,Kukta02-II}.

To formally obtain a  PDE describing the surface dynamics, we briefly recall  some tools from hydrodynamic flows in statistical mechanics.  Setting the atomic volume equal to one, the step chemical potential is given by first variation of the energy \cite{Shenoy02}
\begin{equation*}
\mu_s=\frac{\delta E}{\delta h}=- \Delta_1 h, \qquad \text{with}\ \Delta_1 h:= \nabla \cdot \left(\frac{\nabla h}{|\nabla h|}\right).
\end{equation*}
By the Gibbs-Thomson relation~\cite{Rowlinson,Sethna96,DM_Kohn06,KDM} (which is related to an ideal gas law approximation), the corresponding local-equilibrium density of adatoms is formally 
$\varrho_s = \varrho^0 \exp[\mu_s/(k_B T)]$, where $\varrho^0$ is a constant reference density \cite{Zangwill91,IhleMisbahP-L98}, $T$ is a temperature, and $k_B$ is the Boltzmann constant. An application of Fick's law then predicts that the flux is
\begin{equation*}
\mathbf J=-D_s\,\nabla \varrho_s=-D_s\varrho^0\nabla e^{\mu_s/(k_B T)}~,
\end{equation*}
where $D_s$ is the surface diffusion constant  \cite{DM_Kohn06}. In this way, we obtain the hydrodynamic equation
\begin{equation*}
\partial_t h + \grad \cdot \mathbf J = 0 .
\end{equation*}
Normalizing all constants to be one by rescaling in space and time, we formally arrive at the following PDE  for the evolution of the crystal surface height:
\begin{align}\label{firstcrystalPDE}
\partial_t h=\Delta e^{-\Delta_1 h}\,. % \ , \ \Delta_1 h =  \nabla \cdot \left(\frac{\nabla h} {|\nabla h|} \right) .
\end{align}
A  thorough derivation of \eqref{firstcrystalPDE} from microscopic dynamics can be found in \cite{LLMM1}.  

Away from facets, this equation is consistent with the continuum limit of the Burton-Cabrera-Frank (BCF) theory for moving steps in 2+1 dimensions~\cite{BCF51,DM_Kohn06}.  See also \cite{BonzelPreuss95} for a numerical study of $1d$ facet dynamics.  This equation also relates to a family of Kinetic Monte Carlo models of crystal surface relaxation, including both  the solid-on-solid (SOS) and discrete Gaussian models, in which the 1-Laplacian is replaced by a $p$-Laplacian, $p \geq 1$~\cite{KDM,MW1,gao2019analysis}.

Note that, even in one dimension and for $h(x,t)$ is smooth, the  1-Laplacian $\Delta_1 h$ is a linear combination of positive and negative Dirac masses, so $e^{-\Delta_1 h}$ is not well-defined. Consequently, equation (\ref{firstcrystalPDE}) must be interpreted in a generalized sense. One avenue considered in previous work is to take a first order approximation of the exponential  in the Gibbs-Thomson relation, replacing $e^x$ with $1 +x$, which leads to the  $H^{-1}$ total variation flow studied by Giga, et. al., \cite{KobayashiGiga99,GigaGiga,GigaKohn,Giga1,Odisharia_06} 
\begin{align}
\label{eqn:tvflow}
\partial_t h =    \Delta ( - \Delta_1 h)~\,.
\end{align} 
A limitation of this approach is that it treats local maxima and minima of $h$ symmetrically, in contrast to the original equation (\ref{firstcrystalPDE}), which causes local maxima to form expanding facets, while local minima remain stationary.     Ultimately, determining an appropriate notion of weak solution for     equation (\ref{firstcrystalPDE}) and proving existence of   solutions remains a challenging open problem. 

In spite of these gaps in the  underlying theory of the crystal surface evolution equation, we seek to develop a computationally efficient numerical method for accurate simulation of its solutions, while respecting the inherent asymmetry between facet formation at local maxima and pinning at local minima. Recent work by the middle three authors and Margetis\cite{LLMM1} and the fourth author and Weare \cite{MW1}, numerically explored the crystal surface evolution equation, using various regularizations. On one hand, these simulations compared well with existing microscopic models and respected the different dynamics at  local  maxima and minima. On the other hand, they were not motivated by a strong notion of convergence to the macrosopic PDE dynamics, and due to the inherent stiffness of the  model, were only effective on coarse spatial grids, with serious numerical convergence issues arising on fine grids, even in one dimension.  

In contrast, we construct our numerical method for crystal surface evolution by starting with the macroscopic PDE (\ref{firstcrystalPDE}) and leveraging the formal gradient flow structure of the equation, with respect to weighted $H^{-1}$ norms. To see this structure, note that equation (\ref{firstcrystalPDE}) may be rewritten in the following conservative form,
\begin{align} \label{conservativeform0}
\partial_t h + \grad \cdot \left(M(h) \grad \frac{\partial \E}{\partial h } \right) =0 ,
\end{align}
where $\E$ is the total variation energy (\ref{firsttvenergy}) and $M(h)$ is the exponential mobility
\begin{align} \label{firstmobility}
M(h) &:= e^{ - \Delta_1 h} .
\end{align}
For simplicity in what follows, we suppose that our underlying domain is the $d$-dimensional torus $\TT^d$ and equation (\ref{conservativeform0}) is posed with periodic boundary conditions. We normalize the initial data $h(x,0) = h_0(x)$ to have mean zero, $\int h_0 = 0$, a property that is then propagated along the flow (\ref{conservativeform0}).

 Equations of this form (\ref{conservativeform0}) have a formal gradient flow structure with respect to an   $H^{-1}$ norm  weighted by the mobility $M(h)$, which we describe in detail in section~\ref{formalH-1}. For example, choosing the constant mobility $M(h) \equiv 1$, one recovers  classical $H^{-1}$ gradient flows, and in the case of the linear mobility $M(h) = h+1$, one recovers  2-Wasserstein gradient flows on the space of probability measures \cite{O, AGS}. (Since $h$ has mean zero, $h+1$ is a probability density as long as $h \geq -1$.). There has also been significant work  on equations of this form in the context of reaction diffusion equations  \cite{liero2013gradient} and Cahn-Hilliard equations  \cite{lisini2012cahn}, among many others.

Again, the problem of exponentiating $-\Delta_1 h$ arises in the definition of the mobility (\ref{firstmobility}). In order to circumvent this difficulty and thereby ensure that the weighted $H^{-1}$ gradient flow structure is well-defined, we introduce the following novel approximation: given $\varphi \in C_c^\infty(\TT^d)$, $\varphi \geq 0$, $\int_{\TT^d} \varphi = 1$,  $\varphi_\epsilon(x) := \varphi(x/\epsilon)/\epsilon^d$, we consider
\begin{align} \label{regularizedmobility0}
M_\epsilon(h) := e^{- \varphi_\epsilon* \Delta_1 h  } .
\end{align}
 Unlike previous approximations of $e^{-\Delta_1 h}$ via $1-\Delta_1 h$, our approximation respects the inherent asymmetry  near local maxima and minima of $h$, becoming large when  $- \Delta_1 h \gg 0$ and vanishing when $- \Delta_1 h \ll 0$.  

With this approximation in hand, we are able to precisely define the weighted $H^{-1}$ gradient flow of the total variation energy $\E$ with mobility $M_\epsilon$. Then, with the goal of computing this flow numerically, we  discretize the gradient flow in time, with a fixed time step $\tau >0$, via the following semi-implicit method:
\begin{equation} \label{semiimplicitdef0}
h^{n+1} \in \argmin_{h   }\, \E(h)+\frac{1}{2\tau} \|h-h^{n }\|^2_{H^{-1}_{h^{n }}} .
\end{equation}
This approach is inspired by   the  classical minimizing movements scheme for gradient flows, known as the JKO scheme in the 2-Wasserstein context \cite{AGS,JKO}.   In this way, our numerical method can be seen as an extension of recent literature using minimizing movement schemes to simulate nonlinear PDEs as gradient flows on metric spaces; see  \cite{carrillo2019primal, li2020fisher,cances2019variational,carlier2017convergence} and the references therein. More generally, it builds on the well-known literature using implicit Euler time discretizations to simulate Hilbertian gradient flows, including the $H^{-1}$ total variation flow mentioned   in equation \eqref{eqn:tvflow} above \cite{KohnV}. 

We show that the Euler-Lagrange equation characterizing solutions of the  semi-implicit scheme is a discrete time version of the conservative PDE (\ref{conservativeform0}):
 \begin{align} \label{firstEL}
\frac{h^{n+1} -h^n}{\tau} = - \grad \cdot \left( M(h^n) \grad \frac{\partial \E  \ \ \ \ }{\partial h^{n+1}} \right) . %&\iff  \Delta_{h^n}^{-1} h^{n+1} + \tau \grad \frac{\partial \E  \ \ \ \ }{\partial h^{n+1}} =       \Delta_{h^n}^{-1} h^n \\
%&\iff  h^{n+1} =  (\Delta_{h^n}^{-1}   + \tau   \Delta_1)^{-1} \left(    \Delta_{h^n}^{-1} h^n \right)
\end{align}
(See equations (\ref{formalgradientcomputation}) and (\ref{ELeqn}) below.) Consequently, interpolating in time,
\begin{equation*}
  h^{\tau}(x,t) =  h^n, \quad \mbox{ if } t \in [ n\tau, (n+1) \tau)~,
\end{equation*}
and sending our regularization $\epsilon$ and time step $\tau$ to zero, one formally expects that $h^{\tau}(x,t)$ approaches a solution of the crystal surface evolution equation (\ref{firstcrystalPDE}). We leave   analysis of this convergence   to future work, since it directly relates to the challenging open problem of proving existence of solutions to the crystal surface   equation. Still, we believe that the success of our numerical method, which is based on this semi-implicit scheme, provides empirical evidence that the gradient flow framework is the appropriate setting for studying generalized solutions to this equation.

In order to translate the semi-implicit scheme (\ref{semiimplicitdef0}) into a fully discrete numerical method, we use a primal dual hybrid gradient (PDHG) \cite{
  chambolle2016ergodic} approach, which we describe in detail in section \ref{PDHGsection}. This approach allows us to handle the presence of the 1-Laplacian, as well as preserve the energy decreasing property at the discrete level. Given the a step of the semi-implicit scheme $h^n$, our PDHG method iteratively defines a new  sequence $h^{(m)}$ that is initialized at $h^n$ and converges to $h^{n+1}$. A key point in the definition of our PDHG method is that we  use  different norms to penalize the primal and dual variables. As discovered by  Jacobs, L\'eger, Li, and Osher
\cite{jacobs2019solving}, appropriate selection of the norms is
essential to obtaining a scheme that is convergent at the spatially continuous level and leads to a fully discrete numerical method with a rate of convergence that does not deteriorate as the spatial discretization is refined; see Remark
\ref{normchoice}. Our PDHG method is well-defined for the total variation energy $\E$ (\ref{firsttvenergy}) and \emph{any} integrable mobility $M(h)$, including the regularized exponential mobility (\ref{regularizedmobility0}). Provided that the mobility and its reciprocal remain integrable along the sequence $h^n$, which holds for the regularized exponential mobility, our main convergence result Theorem \ref{mainconvresult} proves that the inner PDHG interates $h^{(m)}$ converge to a solution of the outer scheme $h^{n+1}$. We prove this result in one spatial dimension, which coincides with the context of our  numerical simulations. Furthermore, if our initialization $h^n$ has the regularity $( h^n)' \in BV(\TT)$, our theorem provides a  rate of convergence for the PDHG method. We remark that existing  convergence results  for PDHG algorithms do not apply in our context, since our initialization of the primal variable $h^n$ is, in general, infinite $\dot{H}^1$ distance from the optimizer $h^{n+1}$ \cite{chambolle2016ergodic, shefi2014rate}. We instead build upon the approach   introduced by Jacobs, L\'eger, Li, and Osher for the Rudin-Osher-Fatemi image denoising model  \cite{jacobs2019solving, rudin1992nonlinear}. 

Finally, in section \ref{numericalmethodsec}, we use this time discretization of the weighted $H^{-1}$ gradient flow as the basis for a fully discrete numerical scheme,  replacing the spatially continuous operators in our PDHG method with their finite difference  counterparts. In Remark \ref{discreteconvergencermk}, we  describe how the convergence of this fully discrete scheme, with a rate independent of the spatial discretization, follows from similar arguments as given in   Theorem \ref{mainconvresult}. The importance of achieving convergence rates, uniform in the spatial discretization,  was illustrated in several examples by Jacobs, L\'eger, Li, and Osher
\cite{jacobs2019solving}. In the context of our problem, this is even more essential. Convexity properties of  $\| \cdot \|_{H^{-1}_{h^n}}^2$ depend on lower bounds on the  eigenvalues of the weighted Laplacian $\Delta_{h^n}^{-1}$ (see equation \ref{Hm1char}), which may deteriorate along the flow. We are able to cope with this numerically by choosing our inner primal time step in the PDHG method to be relatively large, in agreement with our estimates for the optimal choice in  Theorem \ref{mainconvresult}. This would be impossible with a more classical PDHG method, in which the inner time steps are required to become arbitrarily small  as the spatial discretization is refined.

 We conclude, in section \ref{numericssection}, with several numerical examples that illustrate properties of our method. Our scheme accurately captures facet formation at local maxima and pinning at local minima. Unlike previous numerical methods, which required  a coarse spatial discretization, we observe near first order convergence in both space and time as the spatial discretization and time step $\tau$ are refined. Finally, we also illustrate the importance of norm selection in our PDHG method, showing that selecting norms following the classical $L^2$ approach can cause the number of iterations required for convergence to increase dramatically as the spatial discretization is refined.
 
There are several directions for future work. As mentioned above, we believe that the strength of our numerical method gives hope that the weighted gradient flow setting is the appropriate context in which to define and prove existence of generalized solutions to the crystal surface evolution equation, by analyzing the convergence of the semi-implicit method as $\tau \to 0$ and $\epsilon \to 0$. Our semi-implicit time discretization and PDHG algorithm can also be naturally extended to related crystal evolution PDEs: see Remark \ref{standardlaplacian}, where we describe how the 1-Laplacian in equation (\ref{firstcrystalPDE}) can be replaced by the standard Laplacian. Finally, our convergence result for the PDHG scheme holds for general, integrable mobilities $M(h)$. Consequently, it would be natural to extend our approach  to simulate related gradient flows for other choices of nonlinear mobilities, such as $M(h) = (1+h) (1-h)$ \cite{lisini2012cahn, elliott1996cahn,carrillo2009fermi}.
 
\section{Crystal height evolution as a weighted $H^{-1}$ gradient flow}
\label{formalH-1}

We now describe the weighted $H^{-1}$ gradient flow structure of  the crystal height evolution PDE (\ref{firstcrystalPDE}).
In section \ref{subsec:weight-def}, we define the weighted $H^{-1}$ spaces and the corresponding notions of gradient flow. In section \ref{subsec:Euler-scheme}, we introduce the semi-implicit time discretization of the gradient flow, which is the basis of our numerical scheme. In section \ref{regmob}, we discuss how to apply this framework to the crystal height evolution   equation.% In particular, we discuss some limitations of the mean curvature dependent exponential mobility $M(h)$ in equation (\ref{conservativeform0})---even for smooth functions in one dimension, this requires exponentiation of Dirac masses---and introduce a regularized mobility, with which the weighted gradient flow structure can be made precise. We leave the questions of convergence, for both our time discretization of the gradient flow and our regularization of the mobility, to future work. Our hope is that the framework developed in the present paper will provide the first steps toward the rigorous study of these limits and, ultimately, a proof of existence for solutions to the crystal height evolution equation (\ref{firstcrystalPDE}).

%%%%%%%%%%%%%%%%%%%%%%%%%%%%%%%%%%%%%%%%%%%%%%
\subsection{Weighted $H^{-1}$ gradient flow}
\label{subsec:weight-def}  

For any $h: \TT^d \to \R$, let $M(h) \in L^1(\TT^d)$ denote a nonnegative mobility. Using this   mobility, we  define the   weighted Hilbert space $H^1_h(\TT^d)$ as the completion of
$C^{\infty}(\mathbb T^d)$ functions with mean zero, under  the weighted norm or inner product
\begin{align} \label{norm}
 &   \| v \|^2_{H^1_h} = \int_{\mathbb T^d} M(h) |\nabla v|^2\, \dv x~, \\
 \label{ip}
  &  (u, v)_1 =  \int_{\mathbb T^d} M(h) \nabla u  \cdot  \nabla v \, \dv x~.
\end{align}
We  define $H^{-1}_h(\TT^d) :=\left( H^1_h(\TT^d)\right)^{*}$ to be the dual space of $H^1_h$
 and let $\langle \cdot, \cdot \rangle$ denote the duality pairing. % By our assumption  that elements of $H^1_h(\TT^d)$ have mean zero, elements of $H^{-1}_h$ are indistinguishable up to the addition or subtraction of constants. Constantly, we assume, without loss of generality, that 

By the Riesz-Fr\'echet representation theorem, 
the duality mapping $J : H^1_h \to H^{-1}_h $ given by
\begin{equation*}
 \langle J(v), u \rangle = (v, u)_1, \quad \forall u \in  H^1_h \, 
\end{equation*}
is surjective.
Now, consider the weighted Laplacian operator
\begin{equation*}
 \Delta_h u = \nabla \cdot ( M(h) \nabla u) \, ,
\end{equation*}
which is well defined for $u \in C^\infty(\TT^d)$, in the sense of distributions.
For $u,v \in C^\infty(\TT^d)$ with mean zero,
by definition of  $(\cdot, \cdot)_1$ and integration by parts, we have
\begin{equation*}
   \langle J(v), u \rangle =   (v, u)_1  = - \int_{\mathbb T^d} u  \Delta_{h} v \,\dv x~.
\end{equation*}
Hence, we identify  $J(v) = - \Delta_h v$. 

The inverse map $J^{-1}: H^{-1}_h \to H^1_h,  \phi \mapsto  J^{-1}(\phi)$ is then given
by
\begin{equation*}
  \langle \psi, J^{-1}(\phi)\rangle  = (\psi, \phi)_{-1} 
 = - \int_{\mathbb T^d} \psi ( \Delta_{h}^{-1}\phi)\, \dv x~, \quad \forall \psi \in H^{-1}_h  \, ,
\end{equation*}
where $(\cdot, \cdot)_{-1}$ denotes the inner product for  $H^{-1}_h$ and  $\Delta_h^{-1}$ denotes the inverse operator of $\Delta_h$ with mean zero. Consequently, we obtain,
\begin{equation} \label{Hm1char}
\|\psi\|^2_{H^{-1}_{h}} = - \int_{\TT^d} \psi\Delta_{h }^{-1}\psi \, \dv x~.
\end{equation}

We now turn to the differential structure induced by the $H^{-1}_h$ norm. Given a convex functional  $\E: H^{-1}_h\to \mathbb{R}\cup \{+\infty\}$, its subdifferential is  
$$
   \partial_{H^{-1}_h} \E(\psi) = \left \{ \xi \in    H^{-1}_h(\TT^d) :
   \E(\varphi) \ge \E(\psi) + ( \varphi-\psi ,\xi  )_{-1} \quad \forall \varphi \in H^{-1}_h \right\}~.
$$
For example,   the identity mapping $\psi \mapsto \{ \psi\}$ is the subdifferential of 
the convex functional $\E(\psi)=\tfrac12 \| \psi \|^2_{ H^{-1}_h}$.
%Katy note to self: rather than having the subdifferential be contained in the dual space, I use Riesz Representation to write it as contained in the primal space. This is the standard approach for Hilbert space graident flows and avoids too many applications of the duality mapping.

Using this notion of subdifferential, we may   define $H^{-1}_h$ gradient flows. In order for our construction of the weighted Hilbert spaces to remain valid, we require that $M(h)$ remains nonnegative and integrable along the flow, that is, the flow remains in the space
\begin{align*}
L_M^1  = \left\{ h : \TT^d \to \R : M(h) \in L^1  \text{ and } M(h)  \geq 0 \right\} .
\end{align*}

Next, we introduce a notion of time derivative for a flow $h(t)$ evolving through the Hilbert spaces $H^{-1}_{h(t)}$.
\begin{definition} \label{absolutely continuous}
Given $h: [0,T] \to L^1_M$ such that $h(t) \in H^{-1}_{h(t)}$ for all $t \in [0,T]$, we say that \emph{$h(t)$ is differentiable with respect to  $\| \cdot \|_{H^{-1}_{h(t)}}$} in case, for all $t \in [0,T]$, there exists $\epsilon >0$ so that, for all $s \in (t-\epsilon , t+\epsilon) \cap [0,T]$,  $h(s) \in H^{-1}_{h(t)}$ and $h(s)$ is Fr\'echet differentiable with respect to $\| \cdot \|_{ H^{-1}_{h(t)}}$ .\end{definition}

With this, we can now define an $H^{-1}_{h(t)}$ gradient flow.
\begin{definition} \label{GFdef}
Given    $h:[0,T] \to L_M^1 $ such that $h(t) \in H^{-1}_{h(t)}$ is differentiable, we say $h$ is an $H^{-1}_{h(t)}$ \emph{gradient flow} of an energy $\E : H^{-1}_{h(t)} \to \R \cup \{ +\infty \}$ with initial condition $h_0$ in case 
\begin{align} \label{generalweightedflow}
\begin{cases}  \partial_t h(t)   \in - \partial_{H^{-1}_{h(t)}} \E(h(t)) \text{ for  all }t \in [0,T], \\
h(0) = h_0 .
\end{cases}
\end{align}
\end{definition}

In particular, given an energy $\E : H^{-1}_h \to \R \cup \{ +\infty \}$, we formally obtain the following expression for its gradient with respect to $H^{-1}_h(\TT^d)$,
\begin{align*}
\lim_{\epsilon \to 0} \frac{ \E(\psi + \varepsilon \xi)  - \E(\psi)}{\varepsilon}  &= \int_{\TT^d}  \frac{\partial E}{\partial \psi}  \xi
= \int_{\TT^d}  \Delta_h^{-1} \Delta_h \frac{\partial E}{\partial \psi}   \xi  = \left( \Delta_h \frac{\partial E}{\partial \psi} , \xi \right)_{-1}
\end{align*}
Therefore,
\begin{align*}
\grad_{H^{-1}_h} \E(\psi) =  \Delta_h \frac{\partial E}{\partial \psi} .
\end{align*}
Consequently, under sufficient regularity of the energy functional $\E$ and under the assumption that the mobility remains integrable and nonnegative along the flow, $H^{-1}_h(\TT^d)$ gradient flows  correspond to solutions of the conservative PDE (\ref{conservativeform0}),
\begin{align} \label{formalgradientcomputation}
 \partial_t h  = - \grad_{H^{-1}_h} \E(h)   \iff \partial_t h + \Delta_h \frac{\partial E}{\partial h} = 0 \iff \partial_t h + \grad \cdot \left(M(h) \grad \frac{\partial E}{\partial h} \right) =0 . \hspace{-.5cm}
\end{align}

%%%%%%%%%%%%%%%%%%%%%%%%%%%%%%%%%%%%%%%%%%%%%%%%%%%%%%%%
\subsection{Semi-implicit   scheme for $H_h^{-1}$ gradient flows}
\label{subsec:Euler-scheme}

We now describe a semi-implicit analogue of the classical minimizing movement scheme   to discretize  our   $H_h^{-1}$ gradient flows in time: given $h^n \in L^1_M  \cap H^{-1}_{h^n} $, solve
\begin{equation} \label{semiimplicitdef}
h^{n+1} \in \argmin_{h \in H^{-1}_{h^n}}\, \E(h)+\frac{1}{2\tau} \|h-h^{n}\|^2_{H^{-1}_{h^{n}}} .
\end{equation}
In the particular case   $M(h) = h+1$, $h \geq -1$,  $H^{-1}_h$ gradient flows are 2-Wasserstein gradient flows, and the above method can be interpreted as a semi-implicit variant of the  Jordan Kinderlehrer Otto (JKO) scheme \cite{JKO}, in which the Wasserstein distance is approximated by the corresponding weighted $H^{-1}$ norm at the previous time step \cite{cances2019variational}.

We begin by showing that, as long as the energy $\E$  is convex, lower semicontinuous, and has compact sublevels with respect to an appropriate topology and  $\E(h^n)<+\infty$, then there exists a unique solution to this semi-implicit scheme.

\begin{proposition} \label{existenceminmov}
Fix $h^n \in   L^1_M \cap H^{-1}_{h^{n}}$ and consider an energy $\E:H^{-1}_{h^n} \to \R \cup\{+\infty\}$. Suppose $\E$  is convex and that there exists a topology $\sigma$ so that $\E$ and $H^{-1}_{h^n}$ are both lower semicontinuous with respect to $\sigma$ and the sublevel sets of $\E$ are relatively $\sigma$-compact in $H^{-1}_{h^n}$.  Then, if $\E(h^n)<+\infty$,  there exists a unique $h^{n+1}$ so that
\begin{align}  \label{hndefphi}
h^{n+1} \in \argmin_{h \in H^{-1}_{h^n}} \, \Phi(h) , \quad \text{ for }\quad \Phi(h) := \E(h)+\frac{1}{2\tau} \|h-h^n\|^2_{H^{-1}_{h^n}} . 
\end{align}
\end{proposition}

\begin{remark}  
Our assumption that $h^n \in L^1_M$, or equivalently, that the mobility $M(h^n)$ is integrable and nonnegative, is necessary for  the weighted Hilbert spaces to be well-defined. Analogous requirements on the mobility have   arisen in recent work by Canc\'es, Gallou\"et, and Todeschi \cite{cances2019variational}, in which they consider a fully-implicit time discretization, in the special case that $M(h) = h+1$ and $h >-1$. \end{remark} 
 
\begin{remark} We choose to introduce the additional topology $\sigma$ in Proposition \ref{existenceminmov} due to the fact that, in general, the topology induced by $H^{-1}_{h_n}$ may not be strong enough to ensure lower semicontinuity of the energy. In particular, this is the case for the exponential mobility and total variation energy we consider in the next section.
\end{remark}

  \begin{proof}[Proof of Proposition \ref{existenceminmov}]
 First, we consider existence.   Since $\Phi(h^n) = \E(h^n)<+\infty$,  
 \[ \inf_{h \in H^{-1}_{h^n}(\TT)}  \Phi(h)   <+\infty , \]
  and we may choose a minimizing sequence $h^k \in H^{-1}_{h^n}(\TT) $ so that  $\lim_{k \to +\infty} \Phi(h^k)= \inf_h \Phi(h)$. Since $\Phi(h) \geq \E(h)$, $\{h^k\}$ belongs to a sublevel set of $\E$, so up to a subsequence, there exists $\bar{h}$ so that $h^k \xrightarrow{\sigma} \bar{h} \in H^{-1}_h$. By lower semicontinuity of $\E$ and $\|\cdot \|_{H^{-1}_{h^n}(\TT)}$ with respect to $\sigma$, $\liminf_{k \to +\infty} \Phi(h^k) = \Phi(\bar{h})$. Thus,  $\bar{h}$ is a solution of (\ref{semiimplicitdef}), so a solution exists.

It remains to show uniqueness. Suppose $h^{n+1}$ and $\bar{h}$ are distinct solutions of (\ref{semiimplicitdef}). Define $h_\alpha = (1-\alpha) h^{n+1} + \alpha \bar{h}$. Then, by the convexity of $\E$ and the strict convexity of $h \mapsto \| h - h^n \|^2_{H^{-1}_{h^n}}$,
\begin{align*}
\Phi(h_\alpha) < (1-\alpha) \Phi(h^{n+1}) + \alpha \Phi(\bar{h}) = \inf_{h \in H^{-1}_h(\TT^d)} \Phi(h) ,
\end{align*}
which is a contradiction. Therefore $h^{n+1} = \bar{h}$, so solutions of  (\ref{semiimplicitdef}) are unique.
 \end{proof}
 
Given a discrete sequence $\{h^n\}$ defined by our semi-implicit scheme (\ref{semiimplicitdef}), the convexity of $\Phi$ and the fact that $h^{n+1}$ is a global minimum implies that we have the following Euler-Lagrange equation characterizing $h^{n+1}$, 
\begin{align} \label{ELeqn} 0 \in \partial_{H^{-1}_{h^n}} \Phi(h^{n+1})  \iff \frac{h^{n+1} -h^n}{\tau} \in -\partial_{H^{-1}_{h^n}} \E(h^{n+1}) .\end{align}
Consequently, interpolating in time,
$h^{\tau}(x,t) =  h^n$, if $t \in [ n\tau, (n+1) \tau)$, 
one formally expects that, under sufficiently regularity of $\E$ and $h^n$, as $\tau \to 0$, $h^{\tau}(x,t)$ approaches a solution of the $H^{-1}_{h^n}$ gradient flow, in the sense of Definition \ref{GFdef}. A key difficulty in the  analysis of this limit is proving that the discrete time solutions $h^n$ remain in the space $L_M^1$, so that the weighted dual Sobolev spaces $H^{-1}_{h^n}$ remain well-defined. This depends strongly on the choice of energy $\E$. We leave the rigorous study of this limit to future work.   Our hope is that the framework developed in the present paper will provide the first steps toward the rigorous study of this limit  and, ultimately, a proof of existence for solutions to the crystal height evolution equation.

%%%%%%%%%%%%%%%%%%%%%%%%%%%%%%%%%%%%%%%%%%%%%%%%%%%
\subsection{$H^{-1}_h$ gradient flow for crystal surface evolution}
\label{regmob}
\ 
We now describe  how our crystal surface evolution equation fits into this  gradient flow framework. As discussed in the introduction,  the crystal surface evolution PDE may be formally rewritten in conservative form (\ref{conservativeform0}) for an  exponential mobility (\ref{firstmobility}) and the total variation energy.
 However, in order for this formal description to coincide with a well-defined $H^{-1}_h$ gradient flow, we must extend the energy to a functional defined on all of $H^{-1}_h$, satisfying the hypotheses of Proposition \ref{existenceminmov},  and the mobility must remain nonnegative and integrable along the flow.  We now consider  each of these issues. 
  \subsubsection{Total variation energy on $H^{-1}_h $}
Since $H^1_h $ is defined as the completion of $C^\infty $ functions with mean zero under the $H^1_h$ norm, any element $\psi \in H^{-1}_h $ is uniquely defined by its action on such smooth functions. In particular, if there exists $f \in L^1 $ with mean zero so that $\la \psi, v \ra = \int_{\TT^d} f v \, \dv x$ for all $v \in C^\infty $ with mean zero, we will identify $\psi$ with $f$ and say $\psi \in L^1 $. (We   restrict to $f$ with mean zero, since such an $f$ is only determined up to a constant.)

In this way, we  extend the definition of the total variation energy to   $H^{-1}_h $,
  \begin{align} \label{energydef}
\E(\psi) := \begin{cases} \| \psi\|_{TV} &\text{ if } \psi \in L^1(\TT^d) \text{ and } \int \psi = 0,  \\
+\infty &\text{ otherwise,}\end{cases}
\end{align}
where, for any $\psi \in L^1(\TT^d)$,
\begin{align} \label{tvdef}
\|\psi\|_{TV} :=  \sup_{\phi } \left\{ - \int_{\TT^d}   \psi \grad \cdot \phi  :  \phi \in C^\infty(\TT^d),  \  \| \phi \|_{\infty} \leq 1    \right\} .
\end{align}
Furthermore, if $\|\psi\|_{TV} < +\infty$, then the distributional derivative $\grad \psi$ is a signed measure and
\begin{align} \label{tvdef2}
\|\psi\|_{TV} =  \sup_{\phi } \left\{ \int_{\TT^d}  \grad  \psi  \cdot \phi  :  \phi \in L^\infty(\TT^d),  \  \| \phi \|_{\infty} \leq 1    \right\} .
\end{align}

We now show that $\E$ satisfies the hypotheses of Proposition \ref{existenceminmov}, so that the semi-implicit scheme is well defined. 
 
\begin{proposition} \label{TVenergyprop}
Consider $h \in L_M^1 $, so that $M(h)$ is nonnegative and integrable, and consider the total variation energy $\E: H^{-1}_h  \to \R \cup \{+\infty\}$ defined in equation (\ref{energydef}). Then $\E$ is convex, and, letting $\sigma$ denote the topology of convergence in distribution,  $\E$ and $H^{-1}_h$ are both lower semicontinuous with respect to $\sigma$ and the sublevel sets of $\E$ are relatively $\sigma$-compact in $H^{-1}_h$.  \end{proposition}

\begin{proof}
The convexity of $\E$ follows immediately from the fact that $L^1 \cap H^{-1}_h $ is convex and $\| \cdot \|_{TV}$ is a convex functional on $L^1$. 

Next, we show that  $\E$ and $\| \cdot \|_{H^{-1}_h}$  are lower semicontinuous with respect to convergence in distribution. 
We begin with  $\| \cdot \|_{H^{-1}_h}$. Suppose $\psi^k \in H^{-1}_h$ converges to $\psi \in H^{-1}_h$ in distribution. Then,
\begin{align*}
\liminf_{k \to \infty} \|\psi^k \|_{H^{-1}_{h }} &= \liminf_{k \to \infty} \sup_{\phi \in C^\infty, \|\phi \|_{H^1_{h }} \leq 1 } \la \psi^k , \phi  \ra\geq \sup_{\phi \in C^\infty, \|\phi \|_{H^1_{h }} \leq 1 }  \liminf_{k \to +\infty} \la \psi^k, \phi \ra \\
&=  \sup_{\phi \in C^\infty, \|\phi \|_{H^1_{h }} \leq 1 }  \la \psi, \phi \ra = \|\psi \|_{H^{-1}_h}
\end{align*}

We now show lower semicontinuity of $\E$  with respect to a sequence $\psi^k \in H^{-1}_h$ converging to $\psi \in H^{-1}$ in distribution. Without loss of generality, we may assume that $\liminf_{k \to \infty} \E(\psi^k) < +\infty$, or the result is trivially true. Consider a subsequence, $\psi^{k^l}$ that attains the limit, i.e. $\liminf_{k \to \infty} \E(\psi^k) = \lim_{l \to _\infty} \E(\psi^{k^l})$ and for which $\E (\psi^{k^l}) < +\infty$. For simplicity of notation, we identify this subsequence with the original sequence $\psi^k$. Since $\E(\psi^k) < +\infty$ for all $k$, along this sequence, the energy coincides with $\| \cdot \|_{TV}$. Thus, 
\begin{multline*} 
  \lim_{k \to \infty} \E(\psi^k) = \lim_{k \to \infty}  \|\psi_k\|_{TV} =  \lim_{k \to +\infty}  \sup_{\phi \in C^\infty, \|\phi \|_\infty \leq 1 }- \int   \psi^k \grad \cdot \phi \\
   \geq \sup_{\phi \in C^\infty, \|\phi \|_\infty \leq 1 }  \lim_{k \to +\infty} -\int \psi^k \grad \cdot \phi 
=\sup_{\phi \in C^\infty, \|\phi \|_\infty \leq 1 }  -\int \psi \grad \cdot \phi = \|\psi\|_{TV} = \E(\psi).
\end{multline*}

We now show relative compactness of the sublevel sets of $\E$.  
Suppose $\sup_k \E(\psi^k) \leq C$ for some $C \in \R$. By classical results, there exists $\psi \in L^1(\TT^d)$ such that $\E(\psi) = \|\psi\|_{TV} \leq C$ and $\psi^k \to \psi$ in $L^1(\TT^d)$ \cite[Corollary 5.3.4]{ziemer2012weakly}. Since convergence in $L^1(\Omega)$ implies convergence in distribution and $\| \cdot \|_{H^{-1}_h}$ is lower semicontinuous in distribution, we conclude $\psi \in H^{-1}_h$, which gives the result. 
 \end{proof}

 \subsubsection{Regularization of mobility} \label{mobregsec}
While we require that the mobility   remain nonnegative and integrable along the flow, in order for the spaces $H^{-1}_h$ to remain well defined, this fails for the exponential mobility (\ref{firstmobility}),
even for smooth functions in one dimension. For example, for any  $h \in C^\infty(\TT)$, the function $ h'/| h'| : \TT \to \{-1, 0,1 \}$ is piecewise constant,  and $-\Delta_1 h = -  (  h' /| h'|)'$ is signed measure, consisting of a linear combination of positive and negative Dirac masses,  corresponding  to local maxima and minima of  $h$. Thus, $e^{ - \Delta_1 h }$ is not well-defined.

Consequently, we instead approximate the mobility by convolving  $- \Delta_1 h $ with a mollifier. Given $\varphi \in C_c^\infty(\TT^d)$, $\varphi \geq 0$, $\int_{\TT^d} \varphi = 1$, define the mollifier $\varphi_\epsilon(x) = \varphi(x/\epsilon)/\epsilon$. We then consider the mobility
\begin{align} \label{regularizedmobility}
M_\epsilon(h) := e^{- \varphi_\epsilon* \Delta_1 h  } ,
\end{align}
which well defined for $h \in C^1(\TT^d)$, since $\varphi_\epsilon* \Delta_1 h = \grad \varphi_\epsilon*(\grad h / |\grad h|) \in C^\infty(\TT^d)$ for all $\epsilon >0$. In one dimension, this regularization replaces each Dirac mass in $\Delta_1 h$ with an appropriately weighted mollifier $\varphi_\epsilon$.   Since $\varphi_\epsilon * \Delta_1 h \to \Delta_1 h$ in the narrow topology as $\epsilon \to 0$, we formally expect that this approximation of the crystal height dynamics converges as $\epsilon \to 0$, but we leave the rigorous analysis of this limit to future work.

\section{A PDHG method for computing the semi-implicit scheme} \label{PDHGsection}
In the previous section, we defined the following semi-implicit scheme for approximating $H^{-1}_h$ gradient flows,
\begin{align}   \label{originalminimizationproblem}
h^{n+1} \in \argmin_{h \in H^{-1}_{h^n}} \, \Phi(h) , \quad \text{ for }\quad \Phi(h) := \E(h)+\frac{1}{2\tau} \|h-h^n\|^2_{H^{-1}_{h^n}} . 
\end{align}
In order to use this scheme as a numerical method for simulating solutions of the crystal growth equation, we need an approach to compute the minimizer $h^{n+1}$ of $\Phi$.

In this section, we   reformulate the above minimization problem  as a saddle-point problem, so that solutions can be
computed via operator splitting methods. In particular, given an
element of the discrete time sequence $h^n$ we apply a primal dual
hybrid gradient (PDHG) method \cite{
  chambolle2016ergodic} to compute the next element in the sequence
$h^{n+1}$. The PDHG method is essentially composed of alternating
implicit Euler steps in the primal and dual variables, subject to
appropriate averaging; see Remark \ref{proxpointinterp}. An
important aspect of our method is that the implicit Euler step in the
primal variables is taken with respect to an $\dot{H}^1$ norm, while
the implicit Euler step in the dual variables is with respect to an
$L^2$ norm. Appropriate selection of the norms is
essential to proving convergence of the scheme; see Remark
\ref{normchoice}. This also leads to a fully discrete numerical method that converges with a rate independent of the   spatial discretization; see Remark \ref{discreteconvergencermk}.

We begin, in section \ref{PDHGdefsec}, by defining our PDHG scheme. In section \ref{convergencesec}, we state our main theorem: in one dimension, provided that the reciprocal of the mobility remains integrable, the PDHG scheme  converges in the ergodic sense to the solution $h^{n+1}$. We   prove this result in section \ref{proofsec}. Our results  apply to the total variation energy (\ref{energydef}) and any nonnegative, integrable mobility $M(h)$.

\subsection{Definition of PDHG scheme} \label{PDHGdefsec}
 
To place our problem in the framework of the PDHG method, note that, by definition of the total variation energy (\ref{energydef}-\ref{tvdef2}), minimizing  $\Phi$ is equivalent to solving the following saddle point problem  
\begin{align} \label{eqn001}
\inf_{h \in L^1 , \int h = 0  } \Phi(h) & \ = 
\inf_{h \in L^1 , \int h = 0 } \  \sup_{  \phi  \in L^\infty } \mathcal{L}(h,\phi) , \\
\mathcal{L}(h,\phi) &:=      \int    \grad h   \cdot \phi +   \frac{1}{2\tau} \|h - h^n \|_{H^{-1}_{h^n}}^2     - F^*(\phi) , \label{lagrangiandef}        \\
 F^* (\phi) &:= \begin{cases} 0 &\text{ if }\| \phi \|_{\infty} \leq 1, \\ +\infty &\text{ otherwise.}\end{cases}
\end{align}
To numerically compute a minimizer of this  problem, we   apply PDHG, initializing the inner  iterations, denoted by $h^{(m)}$, with the value of the semi-implicit sequence at the previous step $h^{(0)} := h^n$ and initializing the dual variables to be zero, $\phi^{(0)} = 0$.  The PDHG algorithm \cite[equation 11]{chambolle2016ergodic} is then given as follows:
\begin{align} \label{cp2p}
 h^{(m+1)} &= \argmin_{h \in L^1 , \int h = 0}  \frac{1}{2\tau} \|h - h^{(0)} \|_{H^{-1}_{h^{(0)}}}^2 + \int  \grad h   \cdot  \phi^{(m)}  + \frac{1}{2\lambda} \|h - h^{(m)}\|_{\dot{H}^1}^2\\
 \bar{h}^{(m+1)} &= 2h^{(m+1)}-h^{(m)}  \label{cp3p} \\
\phi^{(m+1)} &= \argmax_{\phi  \in L^\infty } -F^*(\phi) + \int   \grad \bar{h}^{(m+1)}   \cdot  \phi  - \frac{1}{2 \sigma} \|\phi - \phi^{(m)}\|_{2}^2 ,\label{cp4p}
\end{align}
where $\lambda, \sigma >0$ are given
parameters. %Katy note: we could, without loss of generality restrict to $\phi$ with $\int \phi = 0$, and this wouldn't change the definition of the TV norm, but I don't know if it is necessary.
{We note that the second step is an extrapolation, while the other two
steps are optimization sub-problems in $h$ and $\phi$,
respectively.}

 The PDHG iterations are easier to compute than our original minimization problem (\ref{originalminimizationproblem}), since  their optimizers are characterized by the  Euler-Lagrange equations:
 \begin{align} \label{cp2pp}
 h^{(m+1)} &= \left( -\Delta-\frac{\lambda}{\tau} \Delta_{h^{n}}^{-1}(\cdot - h^{n}) \right)^{-1} \left( - \Delta h^{(m)}  + \lambda \grad \cdot \phi^{(m)} \right)\\
 \bar{h}^{(m+1)} &= 2h^{(m+1)}-h^{(m)}  \label{cp3pp} \\
\phi^{(m+1)} &= (\id + \sigma \partial F^*)^{-1}(\phi^{(m)} + \sigma \grad \bar{h}^{(m+1)} ) \ ,\label{cp4pp}
\end{align}
where 
\begin{align*}
(\id + \sigma \partial F^*)^{-1}(u(x)) = \min(|u(x)|,1) \sgn(u(x)) .
\end{align*}
These have several benefits over  the Euler-Lagrange equation for the semi-implicit scheme (\ref{firstEL}), which in the case of the total variation energy is given by
\begin{align*} h^{n+1} =  (\Delta_{h^n}^{-1}   + \tau   \Delta_1)^{-1} \left(    \Delta_{h^n}^{-1} h^n \right) .
\end{align*}
First, our method allows us to avoid inverting the 1-Laplacian, which would require further regularizations. Second, our approach preserves the decrease of the TV energy at the discrete time level: see Remark \ref{proxpointinterp} and Figure \ref{fig:combined} below. Third, as predicted in our main convergence theorem, Theorem \ref{mainconvresult}, we are able to choose $\lambda$ large to ease inversion of $\Delta_h$: see Figure \ref{fig:xerr} below.
%For further details, see  section \ref{numericalmethodsec}, in which we describe our fully discrete numerical implementation. 

\begin{remark}[interpretation as proximal point algorithm] \label{proxpointinterp}
In the special case that $\lambda = \sigma$, the PDHG method can be characterized as a proximal point algorithm on the product space $\dot{H}^1(\TT) \times L^2(\TT)^d$, endowed with the norm $\| \cdot \|_{\Lmat} := \| \Lmat^{1/2} \cdot \|_{2}$ for
\begin{align*}   \Lmat = \begin{bmatrix} -\Delta &    \lambda \grad \cdot  \\ -\sigma \grad & \id \end{bmatrix}.\end{align*}
For further details in a slightly simpler case see, for example, He and Yuan \cite{he2012convergence}.  
\end{remark}

\begin{remark}[choice of norms] \label{normchoice}
It is essential to the convergence of the PDHG algorithm  that we use a $\dot{H}^1 $ norm penalization in our definition of $h^{(m+1)}$, instead of an $L^2 $ penalization, as in our definition of $\phi^{(m+1)}$. As observed by  Jacobs, L\'eger, Li, and Osher  \cite{jacobs2019solving},  this choice of norms ensures that the gradient operator $\grad: \dot{H}^1 \to (L^2 )^d$ is bounded, so Chambolle and Pock's estimate of the partial primal dual gap applies: see equations (\ref{ppdgdef}) and (\ref{CPineq}) in the proof of our main theorem.
\end{remark}

\begin{remark}[extension to the standard Laplacian] \label{standardlaplacian}
It is possible to extend the above algorithm  to the case of crystal evolution equations with alternative surface energy interactions.  In particular, when $\Delta_1$ is replaced by $\Delta = \Delta_2$ (see e.g. \cite{liu2016existence,liu2017analytical,granero2018global,liu2018global,ambrose2019radius,gao2019analysis,gao2020analysis}), one would replace $F^*(\phi)$ with $F^*(\phi) = \chi_{\|\phi\|_2 \leq 1}$. In this case,  $(\Imat + \sigma \partial F^*)^{-1}(u)  =  u/\|u\|_2.$
On the other hand, for general $\Delta_p$, $p \neq 1,2$, there is no explicit formula for this operator (the proximal map).  
%Note these models are somewhat related to taking the surface energy as
%\begin{align*} 
%\E(h)= \int_\Omega  |\nabla h(x)|^2 \,\dv x .
%\end{align*}
\end{remark}

\subsection{Convergence of PDHG to semi-implicit scheme} \label{convergencesec}
We now prove that, in one dimension, if the reciprocal of the mobility is integrable, we have
  \[ \lim_{M \to +\infty} \Phi( h^{(M)}) = \inf_{ h \in L^1(\TT^d), \int h = 0} \Phi(h) = \Phi( h^{n+1}) ,
\]
where $(h^{(M)}, \phi^{(M)})$ are the ergodic sequences, defined by \begin{align} \label{ergodicseqdef}
  \left( h^{(M)}, \phi^{(M)} \right) = \left( \frac{1}{M} \sum_{m=1}^M h^{(m)},  \frac{1}{M} \sum_{m=1}^M \phi^{(m)} \right) .
  \end{align}
Furthermore, if the initial condition for our PDHG scheme $h^{(0)}:= h^n$ is sufficiently regular, we obtain quantitative estimates on the rate of convergence.

Our main result is the following:  
  \begin{theorem} \label{mainconvresult}
  Suppose the PDHG algorithm is initialized with  
  \begin{enumerate} 
  \item $h^{(0)} :=     h^n \in L^1_M(\TT ) \cap H^{-1}_{h^n}(\TT )$ with $\E(h^{n})< +\infty$ and  $1/M(h^n) \in L^1(\TT)$;   
  \item $\phi^{(0)} \in L^\infty(\TT)$ with $\|\phi^{(0)}\|_\infty \leq 1$ .
\end{enumerate}
Then, for all $\epsilon >0$, there exist $M_*$, $\lambda_*$, $\sigma_*$  so that an $\epsilon$-approximate solution  may be obtained using the step sizes $  \lambda_*$ and $ \sigma_*$ in at most $M_*$ iterations of our scheme, i.e.
  \[ \Phi(h^{(M)}) - \Phi(h^{n+1})  \leq \epsilon \ , \quad \forall M \geq M_* , \]
  where $h^{(M)}$ is the ergodic sequence and  $h^{n+1}$ is the unique minimizer of $\Phi$.
The constants $M_*, \lambda_*,$  $\sigma_*$ depend on $\epsilon$, $\|  h^n\|_{TV}$, $\|M(h^n)\|_1$,  $\|1/M(h^n)\|_1$, and the rate at which the function $\delta \mapsto \|   h^n*\varphi_\delta -   h^n \|_{TV}$ converges to zero, where $\varphi_\delta(x) = \varphi(x/\delta)/\delta$ is a compactly supported mollifier. %\jl{did we specify $\delta$?} {\color{red}jianfeng, is this explanation clearer?}

If, in addition, the initialization $h^{(0)} := h^n$ satisfies $\grad h^n \in BV(\TT),$ then 
\[ \|   h^n*\varphi_\delta -   h^n \|_{TV} \leq \delta \| \grad h^n\|_{TV} M_1(\varphi) \ ,\]
 so there exists 
  a computable constant $c$ depending on $\|h^n\|_{TV}$, $\| \grad h^n \|_{TV}$,  $\|M(h^n)\|_1$, $ \|1/M(h^n)\|_1$ and $\varphi$, so that for  
\begin{align*}
  M_*:= 2 \pi \frac{16  c}{\epsilon^2}, \quad \lambda_* = \frac{c}{\epsilon}, \quad \sigma_* = \frac{\epsilon}{c}  ,
\end{align*}
we have that $(h^{(M)}, \phi^{(M)})$ is an $\epsilon$-approximate solution for all $M \geq M_*$.
  \end{theorem}

 \begin{remark}
The assumption $h^n \in    L^1_M  \cap H^{-1}_{h^n}$, $\E(h^n)<+\infty$ ensures sufficient regularity so that the subsequent step of the scheme $h^{n+1}$ is well-defined; see Propositions \ref{existenceminmov} and \ref{TVenergyprop}.
 \end{remark}
 
\begin{remark}
Our assumption that the reciprocal of the mobility is integrable is similar to analogous assumptions in  recent   work on weighted Hilbert space  discretizations for 2-Wasserstein gradient flows. In particular,   Canc\'es, Gallou\"et, and Todeschi \cite{cances2019variational} consider a fully implicit scheme for $M(h) = h+1 >0$  on a compact domain, which ensures   $1/M(h) \in L^\infty$, hence the reciprocal of the mobility is integrable.

In the particular case of the regularized exponential mobility (\ref{regularizedmobility}),  the constraint that $h^n \in L^1_M$ and $1/M(h^{n}) \in L^1$ is equivalent  to requiring  $M(h^n)$ and $1/M(h^n)$ be integrable. In fact, they are both in $L^\infty(\TT)$ for all $\epsilon >0$, due to the estimate
\begin{align*}
|\grad   \varphi_\epsilon*\sgn(h^n(x))| \leq \frac{1}{\epsilon} \| \grad \varphi \|_{1} .
\end{align*}
\end{remark}

%{\color{blue}  I may want to write the integrability assumption for $1/M(h)$ slightly differently, depending on how I deal with the intervals where $M(h) \equiv 0$. This assumption reminds me of the definition of an $A_2$ weight.}

%We begin by recalling the following characterization of functions of bounded variation.
%\begin{lemma} \label{BVlem}
%Suppose $\int h(x) dx = 0$. Then $h \in BV(\TT)$ if and only if there exists a unique signed measure $\mu$ so that 
%\begin{align} \int_\TT h(x) f'(x) dx = - \int_\TT f(x) d \mu(x) \ \forall f \in C^\infty(\TT) \text{ s.t. } \int_\TT f(x) dx = 0 \ \text{ and } |\mu|(\TT) < +\infty . \label{BVIBP}
%\end{align}
%
%\end{lemma}

 The key step in our proof of Theorem \ref{mainconvresult}, is to estimate
\begin{align} \label{deltarmin}
\min_{\|   h^{(0)} -   h \|_{\dot{H}^1} \leq R} \Phi(h) - \Phi(h^{n+1}) , \quad h^{(0)} := h^n ,
\end{align}
by a quantitative bound that goes to zero as $R \to +\infty$. This is the content of Proposition \ref{objR} below. This estimate shows that, even though the initialization of our PDHG scheme $h^{(0)} = h^n$ will, in general, be    an infinite $\dot{H}^1(\TT)$ distance from the optimizer $h^{n+1}$, we can still make the objective function $\Phi$ arbitrarily close to the optimum while remaining finite $\dot{H}^1(\TT)$ distance from the initialization.

\subsection{Proof of Convergence of PDHG} \label{proofsec}
We begin by   collecting a few basic estimates for the outer semi-implicit time discretization, which are immediate consequences of the definition of the sequence in equation (\ref{semiimplicitdef}), since   $\Phi(h^{n+1}) \leq \Phi(h^{n})$.
 
\begin{lemma}[basic estimates for semi-implicit  scheme] \label{basicminmov} 
Let $\E$ be the total variation energy (\ref{energydef}-\ref{tvdef}), and suppose   $h^n \in    L^1_M(\TT  ) \cap H^{-1}_{h^n}(\TT )$ $\forall \ n \in \mathbb{N}$ and $\E(h^0)<+\infty$. Then,
\begin{enumerate} 
\item $\|   h^{n+1} \|_{TV} \leq \|  h^n \|_{TV} \leq \dots \leq \|  h^0 \|_{TV} < +\infty$,\label{energydec}
\item $\|h^{n+1} - h^{n}\|_{H^{-1}_{h^{n}}} \leq 2 \tau  \|   h^{n} \|_{TV} \leq 2 \tau \|  h^0 \|_{TV}$ .\label{stepsizebound}
 
\end{enumerate}
 \end{lemma}

Next, we collect a few elementary properties of the space $H^{-1}_h(\TT)$.
\begin{lemma} \label{weightednormchar}
Suppose $h \in L^1_M(\TT)$ and $\psi \in H^{-1}_h(\TT)$. Then there exists $\eta_\psi \in L^1(\TT)$   so that $\| \eta_\psi \|_1 \leq  \| \psi\|_{H^{-1}_h} \|M(h)\|_1 $ satisfying
\[ \la \psi, f \ra = \int_\TT \eta_\psi \cdot \grad f \text{ for all } f \in C^\infty (\TT) \quad \text{ and } \quad
 \| \psi \|_{H^{-1}_h}^2 =  \int_{\TT } \frac{|\eta_\psi|^2}{M(h)} . \] 
\end{lemma}

\begin{proof}
By the definition of $H^{-1}_h$ as the dual of $H^{1}_h$ and the Riesz-Fr\'echet Representation theorem, there exists $\xi_\psi \in H^{1}_h$ so that  
\begin{align} \la \psi, f \ra = \int_\TT M(h(x))  \grad f(x) \cdot \grad \xi_\psi(x) dx   \text{ for all $f \in C^\infty(\TT)$ with mean zero } \label{firstetainfo}
\end{align}
and 
\begin{align}  \|\psi \|_{H^{-1}_h}^2 = \|\xi_{\psi} \|_{H^1_h}^2 = \int_\TT M(h(x)) | \grad \xi_{\psi}(x)|^2 dx   . \label{secondetainfo}
\end{align}
Note that, due to the fact that we may add or subtract a constant from $f$ without modifying $\grad f$, equation (\ref{firstetainfo}) holds for all $f \in C^\infty(\TT)$.

Define $\eta_\psi(x) = \grad \xi_\psi(x) M(h(x))$. Since  $\xi_\psi \in H^1_h$ and   $M(h) \in L^1$,  by H\"older's inequality,
\[ \| \eta_\psi \|_{1}  \leq \| \grad \xi_\psi \sqrt{M(h)} \|_2 \|\sqrt{M(h)} \|_2 \leq \|\xi_\psi \|_{H^1_h} \|M(h)\|_1^{1/2} = \| \psi\|_{H^{-1}_h} \|M(h)\|_1^{1/2} . \]
  Finally, substituting $\eta_\psi$ in  equations (\ref{firstetainfo}) and (\ref{secondetainfo}) above  gives the result.
%Katy note to self: remember that if M(h) \equiv 1 in one dimension that all H^1 functions are continuous and a Dirac mass belongs to H^{-1}! Consequently, in general there is no way to make sense of \eta'(x).
\end{proof}

%If M(h) were bounded below, we could try somethign like this:
%\begin{lemma}[$\Delta_h^{-1}$ bounded operator]
%For any $h$ {\color{blue}conditions?} and any $\phi$ with mean zero over $\TT$,
%\[ \|\Delta_h^{-1} \phi \|_{L^2(\TT)} \leq C \|\phi\|_{L^2(\TT)} \]
%\end{lemma}
%\begin{proof}
%This follows from the fact that, by Poincar\'es inequality
%\[ \|\phi\|_{L^2(\TT)} \leq C \|\grad \phi\|_{L^2(\TT)} \]
%{\color{red}problem here, since $M(h)$ not bounded below?}
%\end{proof}

We will also use the following elementary estimate relating the $L^\infty$ and TV norms.
 \begin{lemma} \label{BVtoBlem}
 If $g \in L^1(\TT)$,  $\int g = 0$, and $\|g\|_{TV}<+\infty$, then $\| g \|_{\infty} \leq \|g\|_{TV}$.
 \end{lemma}
  \begin{proof}
 Since $g \in BV(\TT)$ with $\int_\TT g = 0$,  there exist $x_0, x_1 \in \TT$ such that $g(x_0) \geq 0$ and $g(x_1) \leq 0$. By the characterization of the total variation norm in terms of variations of $g$ over partitions of $\TT$, for any such $x_0$ and $x_1$, we have
\[ |g(x_0)| + |g(x_1)| =  g(x_0) - g(x_1) =  |g(x_0) - g(x_1)| \leq \|   g\|_{TV}    . \] 
Hence $\|g \|_\infty \leq \|g \|_{TV}$.
 \end{proof}

In order to quantify the decay of (\ref{deltarmin}), we construct a competitor $h_\delta$ that satisfies the constraint $\|   h^n -   h_\delta\|_{\dot{H}^1} \leq R$ and for which we can estimate $\Phi(h_\delta) - \Phi(h^{n+1})$ by considering the total variation energy $\E$ and the   norm $h \mapsto \| h - h^{n} \|_{H^{-1}_{h^{n}}}$ separately.% {\color{blue}Need to go through and change all gradients to primes, since my current argument really does use one dimension ... JLM: Should be close to done.}
\begin{lemma}[construction of competitor]  \ \label{maingridestimate}
Let $h^{n+1}$ denote the minimizer of $\Phi$. Then, there exists $h_\delta \in BV(\TT)$ so that
\begin{enumerate} 
\item $ \|     h_\delta - h^{n}  \|_{\dot{H}^1} \leq \frac{2 \sqrt{2 \pi}}{\delta} \|\varphi \|_{\infty} \|   h^{n} \|_{TV}$  \label{distanceinitialization}%(analogous to second part of \cite[Lemma B.1]{jacobs2019solving});
\item $\| h_\delta - h^{n} \|_{H^{-1}_{h^{n}}}^2 - \| h^{n+1} - h^{n} \|_{H^{-1}_{h^n}}^2 \leq 16 \pi \delta M_1(\varphi) \|   h^n\|_{TV}^2     \left\|1/M(h^n) \right\|_{1}$;   \label{energydiff}%(analogous to first part of \cite[Lemma B.1]{jacobs2019solving});
\item $\|  h_\delta\|_{TV} - \| h^{n+1}\|_{TV} \leq \|   h^n*\varphi_\delta -   h^n \|_{TV} $.\label{BVenergybound}
\end{enumerate}
\end{lemma}

\begin{proof}
In order to construct our approximating sequence $h_\delta$, we first prove some basic properties of  $h^{n+1}-h^n$. By Lemma \ref{basicminmov} (\ref{energydec}), we have  $\|   h^{n+1} \|_{TV} \leq \|  h^n \|_{TV}  <+\infty$, so $\|   h^{n+1}-h^n \|_{TV} \leq 2 \|  h^n \|_{TV}$. Furthermore, since $h^{n+1}$ and $h^n$ have mean zero, so does $h^{n+1}-h^n$. Thus, by Lemma \ref{BVtoBlem}, we conclude
\begin{align} \label{philinfty}
\|h^{n+1}-h^n\|_{\infty} \leq \|h^{n+1}-h^n\|_{TV} \leq 2 \|   h^n \|_{TV}.
\end{align}

By Lemma \ref{basicminmov} (\ref{stepsizebound}),  we also have
\begin{align}
 \|h^{n+1}-h^n \|_{H^{-1}_{h^n}} \leq 2 \tau \|   h^n\|_{TV} . \label{phidualbound}
\end{align}
 Therefore, by Lemma \ref{weightednormchar}, for $\psi = h^{n+1}-h^n$, there exists $\eta \in L^1(\TT)$ so that  
 \begin{align} \label{etal1}
 \| \eta\|_1 &\leq \| h^{n+1}-h^n\|_{H^{-1}_{h^n}} \|M(h^n)\|_1 , \\
 \label{etaweakder}
 \la h^{n+1}-h^n, f \ra &= \int \eta  f' \text{ for all } f \in C^\infty(\TT)  , \\
  \| h^{n+1}-h^n \|_{H^{-1}_{h^n}(\TT)}^2 &=   \int  \frac{|\eta|^2}{M(h^n)} .  \label{Shnp1}
 \end{align}
Since  $h^{n+1}-h^n \in L^\infty(\TT)$, equation (\ref{etaweakder}) implies that the distributional gradient $ \eta' \in L^\infty(\TT)$, so by Poincar\'e's inequality, $\eta \in W^{1,\infty}(\TT)$ with  
\begin{align} \label{uniformbdeta} 
\|\eta\|_{\infty} \leq 2 \pi \| \eta' \|_{\infty} = 2 \pi \|h^{n+1} - h^n\|_{\infty} \leq 4 \pi \|h^n\|_{TV}
\end{align}
%|\eta(x)| = \left| \int_{x_0}^x \eta'(y) dy + \eta(x_0) \right| \leq 2 \pi \|g\|_\infty + \frac{\tau}{\pi}  \|   h^n \|_{TV} \|(h^n)\|_1 \leq  \tau \|   h^n \|_{TV} \left( 4 \pi  + \frac{1}{\pi}   \|M(h^n)\|_1 \right). \end{align}

We now use $\eta$ to construct our approximation $h_\delta$.
Fix  a compactly supported mollifier $\varphi: \R \to [0,+\infty)$, $\supp \varphi \subseteq B_{2 \pi}(0)$, and let $\varphi_\delta(x) = \varphi(x/\delta)/\delta$. (This mollifier does not need to coincide with that used to regularize the mobility.) Define 
\begin{align}\label{etadeltadef}
\eta_\delta := \eta *\varphi_\delta   .
\end{align}
so $(h^{n+1}-h^{n})*\varphi_\delta =  \eta_\delta'$.
We then choose our approximation $h_\delta$ to be
\begin{align} \label{deltarcand}
h_\delta = h^n + (h^{n+1}-h^{n})*\varphi_\delta
\end{align}

With this definition of $h_\delta$ in hand, we turn to the proof of item (\ref{distanceinitialization}) above.  By inequality (\ref{philinfty}), we have for all $f \in C^\infty(\TT)$,
\begin{align*} 
\left| \int  f'  (h^n  - h_\delta )   \right| &=  \left| \int  f'  (h^{n+1}-h^{n})*\varphi_\delta  \right| =  \left| \int  (\varphi_\delta*f)'  (h^{n+1}-h^{n})    \right| \\
& \leq \|\varphi_\delta *f\|_{\infty}  \|h^{n+1}-h^{n}\|_{TV}   
 \leq \frac{2 \sqrt{2 \pi}}{\delta} \|\varphi\|_{\infty}   \|h^{n}\|_{TV} \|f\|_{2} .
\end{align*}
This ensures $h^n - h_\delta \in H^1(\TT)$ and implies the bound in item (\ref{distanceinitialization}).

Now, we turn to item   (\ref{energydiff}). First, we estimate the rate at which $\eta_\delta$ converges to $\eta$.
By definition of $\eta_\delta$, the fact   $\|  \eta' \|_\infty = \| h^{n+1}-h^n \|_\infty \leq 2\| h^n\|_{TV}$, and inequality (\ref{philinfty}),
\begin{align} \label{etadelconv}
\left| \eta_\delta(x) - \eta(x) \right| &=  \left| \int_\TT \varphi_\delta(x-y) (\eta(y) - \eta(x)) dy \right|  \leq  \| \eta' \|_\infty \int_\TT \varphi_\delta(x-y) |x-y| dy  \nonumber \\
 &\leq 2 \delta M_1(\varphi) \|   h^n\|_{TV} 
\end{align}
where $M_1(\varphi)$ is the first moment of $\varphi$.

Next, we estimate $\|h_\delta-h^{n}\|_{H^{-1}_{h^n}}$ in term of $\eta_\delta$. By definition,  
\begin{align*}
\|h_\delta-h^{n}\|_{H^{-1}_{h^n}} &=   \|(h^{n+1}-h^n)*\varphi_\delta \|_{H^{-1}_{h^{n}}} = \sup_{f \in C^\infty(\TT) \text{ s.t.}  \int f =0 } \frac{ \int (h^{n+1}-h^n)*\varphi_\delta f}{\|f\|_{H^1_{h^n}}}\\
&= \sup_{f \in C^\infty(\TT)  } \frac{ \int (h^{n+1}-h^n)*\varphi_\delta f}{\|f\|_{H^1_{h^n}}} ,
\end{align*}
where in the last equality, we use that $\int h^{n+1} - h^n = \int (h^{n+1}-h^n)*\varphi_\delta = 0$.  Using that $(h^{n+1}-h^n)*\varphi_\delta = (\eta_\delta)'$,  integrating by parts, and applying H\"older's inequality, we obtain that, for any $f \in C^\infty(\TT)$,
\begin{align*}
 \frac{ \int (h^{n+1}-h^n)*\varphi_\delta f}{\|f\|_{H^1_{h^n}}} & =  - \frac{  \int \eta_\delta   f'}{\|f\|_{H^1_{h^n}}} =  - \frac{  \int \eta_\delta M(h^n)^{-1/2} M(h^n)^{1/2}   f'  }{ \left( \int M(h^n) |  f'|^2 \right)^{1/2}}  \\
 & \leq \left( \int \frac{|\eta_\delta|^2}{M(h^n)} \right)^{1/2},
\end{align*}
Thus, $\|h_\delta-h^{n}\|_{H^{-1}_{h^n}} \leq \left( \int |\eta_\delta|^2/M(h) \right)^{1/2}$.

We    apply this     to prove item  (\ref{energydiff}). By equation (\ref{Shnp1}),
\begin{align}
\|h_\delta-h^{n}\|_{H^{-1}_{h^n}}^2 - \|h^{n+1}-h^{n}\|_{H^{-1}_{h^n}}^2 &=  \|(h^{n+1}-h^n)*\varphi_\delta \|^2_{H^{-1}_{h^{n}}} -   \| h^{n+1}-h^{n} \|^2_{H^{-1}_{h^{n}}} \nonumber \\
&\leq  \int  \frac{|\eta_\delta|^2}{M(h^n)} -   \int  \frac{|\eta|^2}{M(h^n)} \nonumber \\
&=   \int  \frac{1}{M(h^n)}  \left( \eta_\delta  -\eta  \right) \left( \eta_\delta  + \eta  \right)    \nonumber  \\
&\leq  \|\eta_\delta - \eta\|_{\infty} \left( \| \eta_\delta\|_\infty + \|\eta\|_\infty \right)   \|1/M(h^n) \|_{1}  \nonumber  \\
&\leq 16 \pi \delta M_1(\varphi) \|   h^n\|_{TV} ^2     \left\|1/M(h^n) \right\|_{1},  \nonumber 
\end{align}
where, in the last inequality, we apply our uniform bound on $\eta$, inequality (\ref{uniformbdeta}), and our uniform estimate on the convergence of $\eta_\delta$ to $\eta$, inequality (\ref{etadelconv}).   This completes the proof of (\ref{energydiff}).

We conclude by  showing item (\ref{BVenergybound}). By the triangle inequality,
\begin{align} \label{lastBVenergybound}
\|  h_\delta\|_{TV} - \|  h^{n+1}\|_{TV} &= \|  h^n + (h^{n+1}-h^n) * \varphi_\delta\|_{TV} - \|  h^{n+1}\|_{TV} \\
&\leq \|   h^n*\varphi_\delta -   h^n \|_{TV} + \|  h^{n+1} *\varphi_\delta \|_{TV} - \|  h^{n+1}\|_{TV}  . \nonumber
\end{align}
Furthermore, for any $f \in C^\infty(\TT)$,
\begin{align*}
-\int  f' (h^{n+1}*\varphi_\delta) = -\int  (\varphi_\delta *f)' h^{n+1} \leq \|\varphi_\delta*f \|_\infty \|h^{n+1}\|_{TV} \leq \|f\|_{\infty} \|h^{n+1}\|_{TV} .
\end{align*}
Therefore $\|  h^{n+1} *\varphi_\delta \|_{TV} \leq \|  h^{n+1}\|_{TV}$, which combined with  (\ref{lastBVenergybound}) gives item (\ref{BVenergybound}).
\end{proof}

We now apply this lemma to prove our key estimate, quantifying the rate of convergence of functions $h$ that are a finite distance from the initialization $h^n$ in the $\dot{H}^1$ norm to the optimizer of $\Phi$.
\begin{proposition} \label{objR} For any  compactly supported mollifier $\varphi_\delta$, there exists an explicit constant $C>0$ depending on  $\|  h^n\|_{TV}$, $\|M(h^n)\|_1$, and $\|1/M(h^n)\|_1$, so that 
\begin{align} \label{deltarmin2}
\min_{\|   h^n -   h \|_{\dot{H}^1} \leq R} \Phi(h) - \Phi(h^{n+1}) \leq \frac{C}{R}+  \|   h^n*\varphi_\delta -   h^n \|_{TV} 
\end{align}
where $\delta = \frac{2 \sqrt{2\pi}}{R} \|\varphi \|_{\infty} \|   h^n \|_{TV}^2$.
\end{proposition}
 
\begin{proof}
Let $h_\delta$ be as in Lemma \ref{maingridestimate} and choose $\delta>0$  so that 
\begin{align} \label{deltaR}
\frac{2 \sqrt{2 \pi}}{\delta} \|\varphi \|_{\infty} \|   h^n \|_{TV} = R.
\end{align} Then Lemma \ref{maingridestimate} (\ref{distanceinitialization}) guarantees that $\|      h_\delta - h^n \|_{\dot{H}^1} \leq R$, so $h_\delta$ is a candidate for the minimization problem (\ref{deltarmin2}). Therefore, it suffices to bound the objective functional when $h = h_\delta$.
By Lemma \ref{maingridestimate} (\ref{energydiff}) and (\ref{BVenergybound}), we have
\begin{multline*}
\left( \frac{1}{2 \tau} \| h_\delta - h^{n} \|_{H^{-1}_{h^n}}^2 + \|  h_\delta \|_{TV} \right) - \left( \frac{1}{2 \tau} \| h^{n+1} - h^{n} \|_{H^{-1}_{h^n}}^2 + \|  h^{n+1} \|_{TV} \right) \\
 \leq 16 \pi \delta M_1(\varphi) \|   h^n\|_{TV} ^2 \|1/M(h^n)\|_{1} + \|   h^n*\varphi_\delta -   h^n \|_{TV}
\end{multline*}
which, combined with equation (\ref{deltaR}), gives the result.
\end{proof}
 
We now turn to the proof of our main result, Theorem \ref{mainconvresult}, which shows that  the PDGH algorithm converges to the optimizer in the ergodic sense: that is, if  $h^{(M)}$ is the ergodic sequence (\ref{ergodicseqdef}), then $\lim_{M \to +\infty} \Phi(h^{(M)}) = \Phi(h^{n+1})$.

  \begin{proof}[Proof of Theorem~\ref{mainconvresult}] %\jl{need to make the superscript notation consistent in the proof}
Following Chambolle and Pock \cite[equation 17]{chambolle2016ergodic} and Jacobs, L\'eger, Li, and Osher \cite{jacobs2019solving}, we consider the  partial primal-dual gap  
\begin{align} \label{ppdgdef}
\mathcal{G}_{R_1,R_2}(h,\phi) &:= \sup_{\hat{\phi} : \left\|   \hat{\phi} -   \phi^{(0)} \right\|_{2} \leq R_1} \mathcal{L}(h,\hat{\phi}) - \inf_{\hat{h} : \left\|   \hat{h} -   h^0 \right\|_{\dot{H}^1} \leq R_2} \mathcal{L}(\hat{h}, \phi) 
\end{align}
where $\mathcal{L}(h, \phi)$ is the Lagrangian defined in equation
(\ref{lagrangiandef}). Since the gradient operator
$\partial_x : \dot{H}^1(\TT) \to L^2(\TT)$ satisfies
\begin{align*}
\| h' \|_{2} = \| h \|_{\dot{H}^1} , \quad \forall h \in \dot{H}^1,
\end{align*}
the operator norm of the gradient is one. Consequently, by
Chambolle and Pock \cite[Theorem 1]{chambolle2016ergodic}, if $\lambda \sigma \leq 1$, then along the ergodic sequences (\ref{ergodicseqdef}),  
\begin{align} \label{CPineq}
\mathcal{G}_{R_1,R_2}(h^{(M)}, \phi^{(M)}) \leq \frac{1}{M} \left(\frac{ R_1^2}{\sigma} + \frac{R_2^2}{\lambda} \right) .
\end{align}

We seek to bound each term in the partial primal-dual gap separately. Since $\| \phi^{(0)}\|_{\infty} \leq 1$ (in fact, in practice we take $\phi^{(0)} = 0$)  we have 
\begin{align} \label{linftytol2} \{ \hat{\phi} : \| \hat{ \phi}\|_{\infty} \leq 1 \} \subseteq \{ \hat{\phi} : \| \hat{\phi} - \phi^{(0)} \|_{2} \leq 2 \sqrt{ 2 \pi} \}. 
\end{align}
 Since $\hat{\phi} \mapsto \mathcal{L}(h^{(M)}, \hat{\phi}) \neq -\infty$ only if $\| \hat{\phi} \|_{\infty} \leq 1$, this implies
\begin{align} \label{supLbound}
  \sup_{\hat{\phi} : \left\|   \hat{\phi} -   \phi^{(0)} \right\|_{2} \leq 2 \sqrt{2 \pi}} \mathcal{L}(h^{(M)},\hat{\phi}) =  \sup_{\hat{\phi} : \| \hat{\phi} \|_{\infty} \leq 1} \mathcal{L}(h^{(M)},\hat{\phi})   =   \Phi(h^{(M)}).
\end{align}

By definition of $\phi^{(m+1)}$ in  equation (\ref{cp4p}), $F^*(\phi^{(m+1)}) < +\infty$, so $\|\phi^{(m+1)}\|_{\infty} \leq 1$  for all $m \in \mathbb{N}$ and the   ergodic sequence also satisfies $\| \phi^{(M)} \|_\infty \leq 1$ for all $M \in \mathbb{N}$. Thus, for any $R>0$,
\begin{align} \label{infLbound}
 \inf_{\hat{h} : \left\|  \hat{h} -  h^n \right\|_{\dot{H}^1} \leq R} \mathcal{L}(\hat{h}, \phi^{(M)})  \leq  \inf_{\hat{h} : \left\|   \hat{h} -   h^n \right\|_{\dot{H}^1} \leq R} \sup_{\phi : \|\phi\|_{\infty} \leq 1} \mathcal{L}(\hat{h}, \phi )  =  \inf_{\hat{h} : \left\|   \hat{h} -   h^n \right\|_{\dot{H}^1} \leq R} \Phi(\hat h) . \end{align}

 Combining these estimates, we conclude that for any $R>0$,
\begin{align*}
\Phi (h^{(M)})  - \Phi(h^{n+1}) & \stackrel{\eqref{supLbound}}{=} \sup_{\hat{\phi} : \left\|   \hat{\phi} -   \phi^n \right\|_{2} \leq 2 \sqrt{2 \pi}} \mathcal{L}(h^{(M)},\hat{\phi})  - \Phi(h^{n+1}) \\ %&\text{ by equation (\ref{supLbound})} \\
&\stackrel{\eqref{ppdgdef}}{=} \mathcal{G}_{2 \sqrt{2 \pi}, R}(h^{(M)}, \phi^{(M)})  + \inf_{\hat{h} : \left\|   \hat{h} -   h^n \right\|_{\dot{H}^1} \leq R} \mathcal{L}(\hat{h}, \phi^{(M)}) - \Phi(h^{n+1}) \\ %&\text{ by equation (\ref{ppdgdef})} \\
& \stackrel{\eqref{infLbound}}{\leq} \mathcal{G}_{2 \sqrt{2 \pi}, R}(h^{(M)}, \phi^{(M)}) +  \inf_{\hat{h} : \left\|   \hat{h} -   h^n \right\|_{\dot{H}^1} \leq R} \Phi(\hat h)- \Phi(h^{n+1}) \\ %&\text{ by inequality (\ref{infLbound})}\\
 & \stackrel{\eqref{CPineq}}{\leq} \frac{1}{M} \left(\frac{ 8 \pi}{\sigma} + \frac{R^2}{\lambda} \right) +  \inf_{\hat{h} : \left\|   \hat{h} -   h^n \right\|_{\dot{H}^1} \leq R} \Phi(\hat h)- \Phi(h^{n+1}) \\ % &\text{ by inequality (\ref{CPineq})} \\
  & \stackrel{\eqref{deltarmin2}}{\leq} \frac{1}{M} \left(\frac{ 8 \pi}{\sigma} + \frac{R^2}{\lambda} \right) + \frac{C}{R}+  \|   h^n*\varphi_\delta -   h^n \|_{TV}\,,%  &\text{ by Proposition \ref{objR}},
\end{align*}
where $\delta = \frac{2 \sqrt{2\pi}}{R} \|\varphi \|_{\infty} \|   h^n \|_{TV}$ and $C>0$ depends on $\|  h^n\|_{TV} $, $\|M(h^n)\|_{1}$, and $\|1/M(h^n)\|_{1}$. We may optimize the first term on the right hand side by choosing %\jl{the choice of parameters are corrected; please check}
\[ \sigma = 2 \sqrt{2 \pi} / R, \quad  \lambda = R/2 \sqrt{2\pi}, \]
in which case we obtain
\begin{align} \label{epsacc1}
\Phi(h^{(M)}) -\Phi(h^{n+1})   \leq \frac{4 \sqrt{2 \pi} R}{M}  + \frac{C}{R}+  \|   h^n*\varphi_\delta -   h^n \|_{TV} .
\end{align}
We claim that, since $\| h^n \|_{TV}< +\infty$, 
\begin{align} \label{claimmoll}
\lim_{R \to +\infty} \|   h^n*\varphi_\delta -   h^n \|_{TV}  = 0.
\end{align}
Thus, we conclude the existence of $M_*, \lambda_*, \sigma_*$ such that for all $M \geq M_*$, we have an $\epsilon$-approximate solution.

It remains to prove the claim (\ref{claimmoll}). Note that if $\phi \in C^\infty$ satisfies $\|\phi \|_\infty \leq 1$, then $\|\varphi_\delta*\phi\|_\infty \leq 1 $  for all $\delta >0$. Hence,
\begin{align*}
 \| h^n*\varphi_\delta \|_{TV} &= \sup_{  \| \phi \|_\infty \leq 1} \int -(h^n*\varphi_\delta) \phi' =   \sup_{  \| \phi \|_\infty \leq 1} \int -h^n  (\phi* \varphi_\delta)' \\
 & \leq    \sup_{ \| \phi \|_\infty \leq 1} \int -h^n  \phi' = \| h^n \|_{TV} < +\infty 
 \end{align*}
This shows $\|h^n*\varphi_\delta - h^n\|_{TV} < +\infty$. Hence, for all $\epsilon >0$, there exists $\phi \in C^\infty$ so  
\begin{align} \label{firsthdeltaconv}
\|h^n*\varphi_\delta - h^n\|_{TV} \leq - \int  (h^n*\varphi_\delta - h^n) \phi' + \epsilon = - \int h^n (  \phi * \varphi_\delta - \phi)' +\epsilon .
\end{align}
Since $\phi$ is a smooth function on a compact set, sending $\delta \to 0$, we conclude that $\limsup_{\delta \to 0} \|h^n*\varphi_\delta - h^n\|_{TV} \leq \epsilon$. Since  $\epsilon >0$ was arbitrary, this proves our claim, again using equation (\ref{deltaR}), relating $\delta$ and $R$.

Now, suppose the function $h^n$ satisfies a higher regularity assumption: $(h^n)' \in BV(\TT)$. Following the same argument as in equation (\ref{firsthdeltaconv}), we have
\begin{align*}
\|h^n*\varphi_\delta - h^n\|_{TV}    =  \sup_{ \|\phi\|_{\infty} \leq 1}    \int  (h^n)'   \cdot (  \phi * \varphi_\delta - \phi)   .
\end{align*}
Furthermore,
\begin{align*}
 \int (h^n)'   \cdot (  \phi * \varphi_\delta - \phi)   &= \iint (h^n)' (x) \cdot (\phi(x-y) - \phi(x)) \varphi_\delta(y) dy dx \\
 &=  - \delta \iint_{\TT \times \TT} \int_0^1 ( (h^n)' (x))^t D \phi(x-s y) y \varphi (y) ds dy dx \\
&\leq \delta \|  (h^n)' \|_{TV} M_1(\varphi) \|\phi\|_{\infty}
\end{align*}

As a consequence, equation (\ref{epsacc1}) becomes
\begin{align*}
\Phi(h^{(M)})-\Phi(h^{n+1})  \leq \frac{4 \sqrt{2 \pi} R}{M}  + \frac{C}{R}+ \delta \|  (h^n)' \|_{TV} M_1(\varphi),
\end{align*}
where $\delta = C'/R$, for $C' =  2 \sqrt{2 \pi} \|\varphi \|_{\infty} \|   h^n \|_{TV} $.
Thus, to obtain an $\epsilon >0$ accurate solution, we require
\begin{align*}
M \geq 4 \sqrt{2 \pi } R  \left( \epsilon -\frac{C''}{R} \right)^{-1}  , \quad \text{ for } C'' =C+C' \| ( h^n)'\|_{TV} M_1(\varphi).
\end{align*}

Optimizing in $R \geq 0$, we obtain that for $R= 2C''/\epsilon$, the choices
\begin{align*}
  M_*:= \frac{16 \sqrt{2 \pi}   C''}{\epsilon^2}, \quad \lambda_* =  \frac{C''}{\sqrt{2 \pi}\epsilon}, \quad \sigma_* = \frac{\sqrt{2 \pi}\epsilon}{C''}  ,
\end{align*}
ensure  that $(h^{(M)}, \phi^{(M)})$ is an $\epsilon$-approximate solution for all $M \geq M_*$.
  \end{proof}

  \section{Fully discrete numerical method} \label{numericalmethodsec}
In this section, we describe how the discrete time, spatially continuous PDHG
  algorithm introduced in section \ref{PDHGdefsec} can be implemented
  as a fully discrete numerical method for simulating crystal surface
  evolution.  In one spatial dimension, let $[0,2\pi]$ be the
  computational domain with periodic boundary conditions and $\Delta x$
  and $\tau$ be the spatial grid spacing and outer time step,
  respectively.  Choose $0=x_1< \cdots < x_{N_x} = 2\pi - \Delta x$,
  where $x_j = (j-1) \Delta x$, $\Delta x = \frac{2\pi}{N_x}$.  For
  notational simplicity, let $h$ and $\phi$ denote the
  discrete vector approximations in $\R^{N_x}$ of the corresponding
  functions, 
\[
h = (h_1, \cdots, h_{N_x})^t, \quad \phi= (\phi_1, \cdots \phi_{N_x})^t .
\]
Let $\Dmat$ and $\Amat$ be the matrix approximations of the operators $\nabla $ and $-\Delta_h$, where $\Dmat$ is given by a centered difference method and $\Amat := \Dmat^t \diag(M(h)_1, \dots, M(h)_{N_x}) \Dmat$.

 \begin{algorithm}[h]
   \caption{PDHG for crystal surface evolution \label{alg:Hdotpen}}
\SetAlgoLined
%$n=0$, $h^{n} = h(0)$ \\
\KwIn{${h}^{0}$, $T$, $\tau$, % $\phi^{n}$, %$\text{Iter}_{max}$,
  $\lambda, \sigma$}
%\KwOut{${h}^{n+1}$}
$n = 0$ \\
\While{ $n \tau \leq T$}{
Let $h^{(0)} = h^n$, $\phi^{(0)} = 0$, and $m=0$; \\
\Repeat{\underline{stopping criteria are achieved}}{ 
\  $h^{(m+1)} = (\Dmat^t \Dmat+\frac{\lambda}{\tau} \Amat^{-1}(\cdot -h^n))^{-1} \left( \Dmat^t \Dmat h^{(m)}- \lambda \Dmat^t \phi^{(m)} \right)$, \\
  $\bar{h}^{(m+1)} = 2h^{(m+1)}-h^{(m)} $, \\
  $\phi^{(m+1)} = (\Imat + \sigma \partial F^*)^{-1}(\phi^{(m)} + \sigma D \bar{h}^{(m+1)})$\,, \\
  $m = m+1$,
}
\ ${h}^{n+1}=\bar{h}^{(m+1)}$ and $n = n+1$, 
}
\end{algorithm}

We discretize our PDHG method \eqref{cp2pp}-\eqref{cp4pp} via a  finite difference scheme, replacing the spatially continuous operators with their discrete counterparts. This leads to Algorithm~\ref{alg:Hdotpen}.   
Finally, we construct our numerical solution $h(x,t)$ for the crystal
surface evolution equation by linearly interpolating between the
spatial gridpoints and taking a piecewise constant interpolation
between the outer discrete time sequence $h^n$.

\begin{remark}[Convergence of fully discrete algorithm] \label{discreteconvergencermk}
Using standard estimates relating finite difference operators to their continuum counterparts, one could adapt our main convergence result, Theorem \ref{mainconvresult}, to be a convergence result for the fully discrete PDHG method, which comprise the inner iterations of  Algorithm \ref{alg:Hdotpen}. See, for example, work by Wang and Lucier \cite{wang2011error}, which considers related estimates for the Rudin-Osher-Fatemi image denoising model. 
%Katy note to self: 
%The adaptation of  Lemma \ref{maingridestimate} (\ref{distanceinitialization})  is not too bad, since this can deteriorate with $\delta$ -- it is okay to use standard Sobolev estimates for quadrature. The adaptation of (\ref{BVenergybound}) follows from Lemma 2.2 in F Error Bounds, for the grid spacing $h \leq h_*$ where $h_* \leq C \delta$. It remains to consider (\ref{energydiff})  }
\end{remark}

In practice, to avoid inverting a near-singular matrix in our computation of $h^{(m+1)}$, we compute the inverse operator in the definition of $h^{(m+1)}$ via
\begin{align} \label{hmp1matrixinverse}
 \left(\Dmat^t \Dmat + \frac{\lambda}{\tau} \Amat^{-1}(\cdot -h^n) \right)^{-1} u 
%&= \left( \Dmat^t \Dmat + \frac{\lambda}{\tau} \Amat^{-1} \right)^{-1} \left(u + \frac{\lambda}{\tau} \Amat^{-1} h^n \right) 
= \left(   \frac{\tau}{\lambda} \Amat \Dmat^t \Dmat + \Imat \right)^{-1} \left(   \frac{\tau}{\lambda} \Amat u  + h^n \right) .
\end{align}
On the other hand, in order to compute $\phi^{(m+1)}$, we use the   explicit formula 
\begin{align*}
(\Imat + \sigma \partial F^*)^{-1}(u) &=   \left[ \min(|u_i|,1) \sgn(u_i) \right] ,
\end{align*}
where $u_i$ denotes the $i$th component of the vector $u$. Note that, while other initializations of the dual variable $\phi^{(0)}$ are possible (for example, initializing $\phi^{(0)}$ to coincide with the last value of $\phi^{(m+1)}$ at the previous outer time step), we observe slightly better performance always initializing $\phi^{(0)} = 0$.

We discretize our regularized mobility as follows: %\jl{changed notation here, realized that $\phi$ is of course used already; and $D$ is used for finite difference to be consistent}
\begin{align}\label{MobCalc}
M(h) :=  e^{- \grad \varphi_\epsilon*\sgn(f)} , \quad f = {\rm minmod} \{ D_+ {h}, D_-  { h} \} ,
\end{align}
where $D_\pm$ denotes the forward/backward finite difference
operators. The minimum modulus limiter of the gradient allows us to
respect shock-like objects in the facet formation; see, e.g.,
\cite{Nessyahu:Tadmor}. Heuristically, this enforces the property of
the original, unregularized mobility (\ref{firstmobility}) that
once a region of the crystal surface becomes flat at a location
$x_0$, i.e. $\frac{d}{dx} h(x_0,t) = 0$, the surface   remains
flat at $x_0$. We compute the convolution in (\ref{MobCalc}) via a
fast Fourier
transform. %{\color{purple} Jeremy, do we always take $\epsilon = 0.01-0.05$? .025-.05 for signum}

\begin{figure}[h]
\hspace{2.8cm} \footnotesize $\sgn(x)$ \hspace{4.5cm} $\tan(10x)$ \hspace{2.2cm}   \\
\includegraphics[height=4.2cm, trim={.2cm .6cm .5cm .5cm},clip]{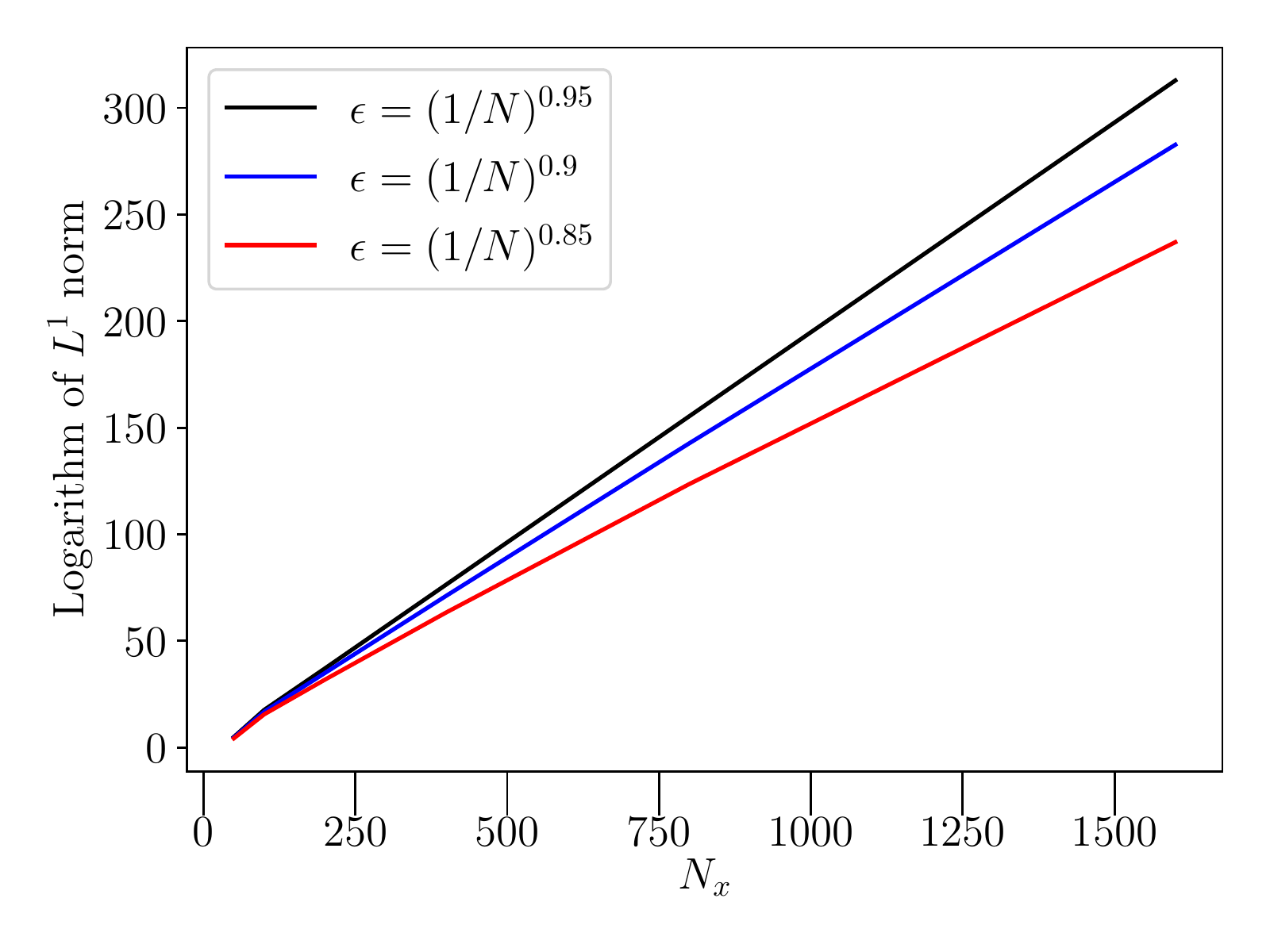} 
\includegraphics[height=4.2cm, trim={.2cm .6cm .5cm .5cm},clip]{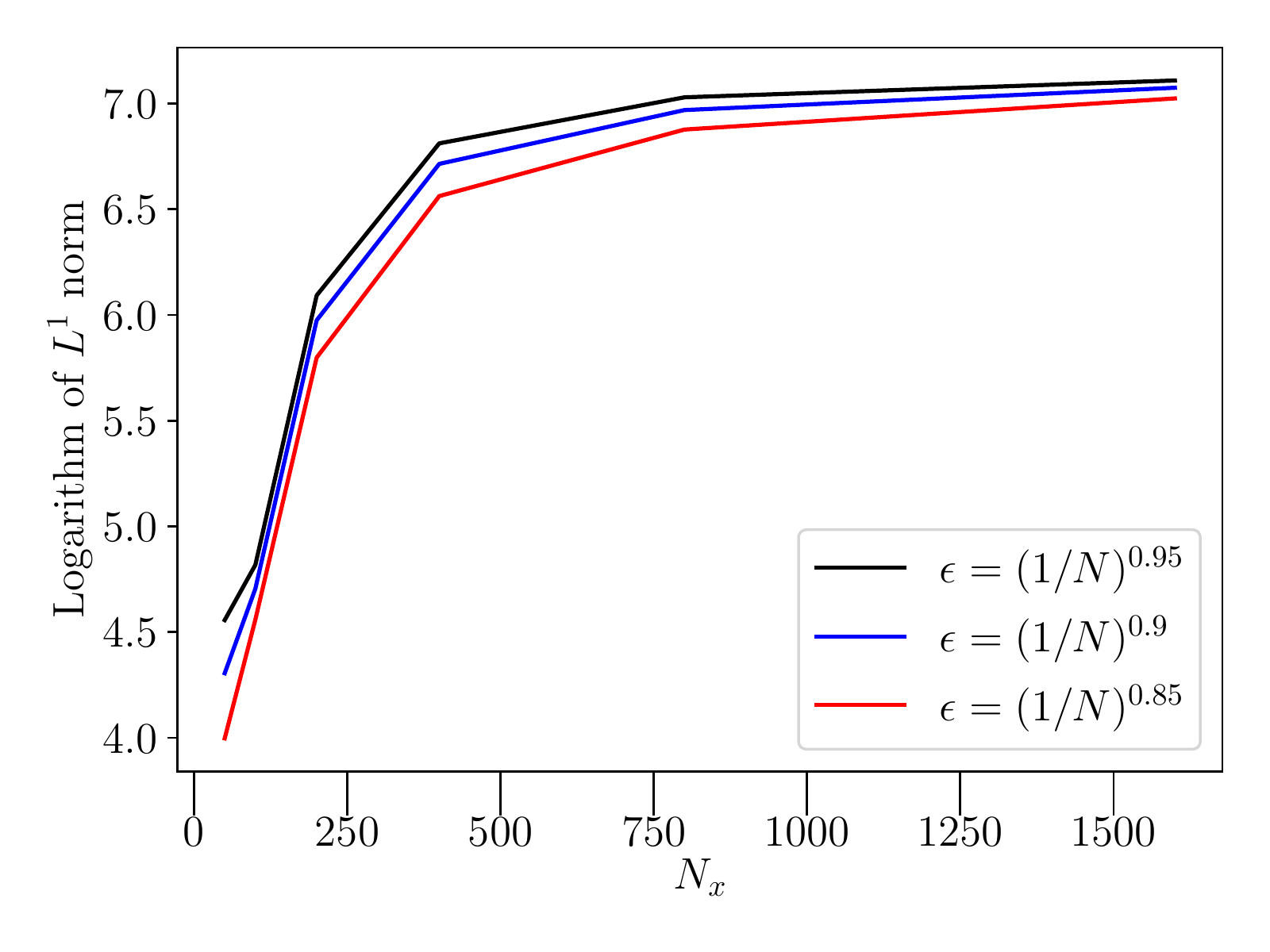}

\caption{Choosing the function $\sgn(x)$ or  $\tanh(x)$ in the   mobility (\ref{MobCalc}) leads to different behavior as $\epsilon \to 0$, $N_x \to +\infty$. Above, we consider the     spatially discrete mobility for height profile $h(x) = \sin(x)$.  Left: For the original mobility, with $\sgn(x)$, even when $\epsilon \to 0$ slowly as $N_x \to +\infty$,  the $L^1$ norm of the mobility diverges. Right: Approximating  with $\tanh(10x)$ allows us to send $\epsilon \to 0$ rapidly as $N_x \to +\infty$, while preserving a uniform bound on the $L^1$ norm of the mobility. }
\label{fig:mob1}
 
\end{figure}

Finally, in our simulations, we  sometimes approximate  $\sgn(x)$ in the definition of the mobility with $\tanh(10 x)$. In order to achieve accurate facet formation, we must strike a balance between choosing the spatial discretization $N_x$ large   and the mobility regularization parameter $\epsilon >0$ small. As illustrated in Figure \ref{fig:mob1}, the original $\sgn(x)$ function is extremely sensitive to small choices of $\epsilon$, which quickly cause the $L^1$ norm of the mobility to become unbounded as $\epsilon \to 0, N_x \to +\infty$, going against the assumption in our convergence result for the PDHG method, Theorem \ref{mainconvresult}, which was proved for fixed $\epsilon >0$. On the other hand, the $\tanh(10 x)$ approximation allows us to refine $\epsilon$ and $N_x$ simultaneously, while keeping the $L^1$ norm of the mobility bounded. A thorough analysis of these  limits is  related to the   question of existence of solutions to the crystal surface evolution equation,    and we leave a detailed study to future work.

\section{Numerical Results} \label{numericssection}
\begin{figure}[h]
\includegraphics[width=4.8cm, trim={.2cm .15cm 1.8cm .8cm},clip]{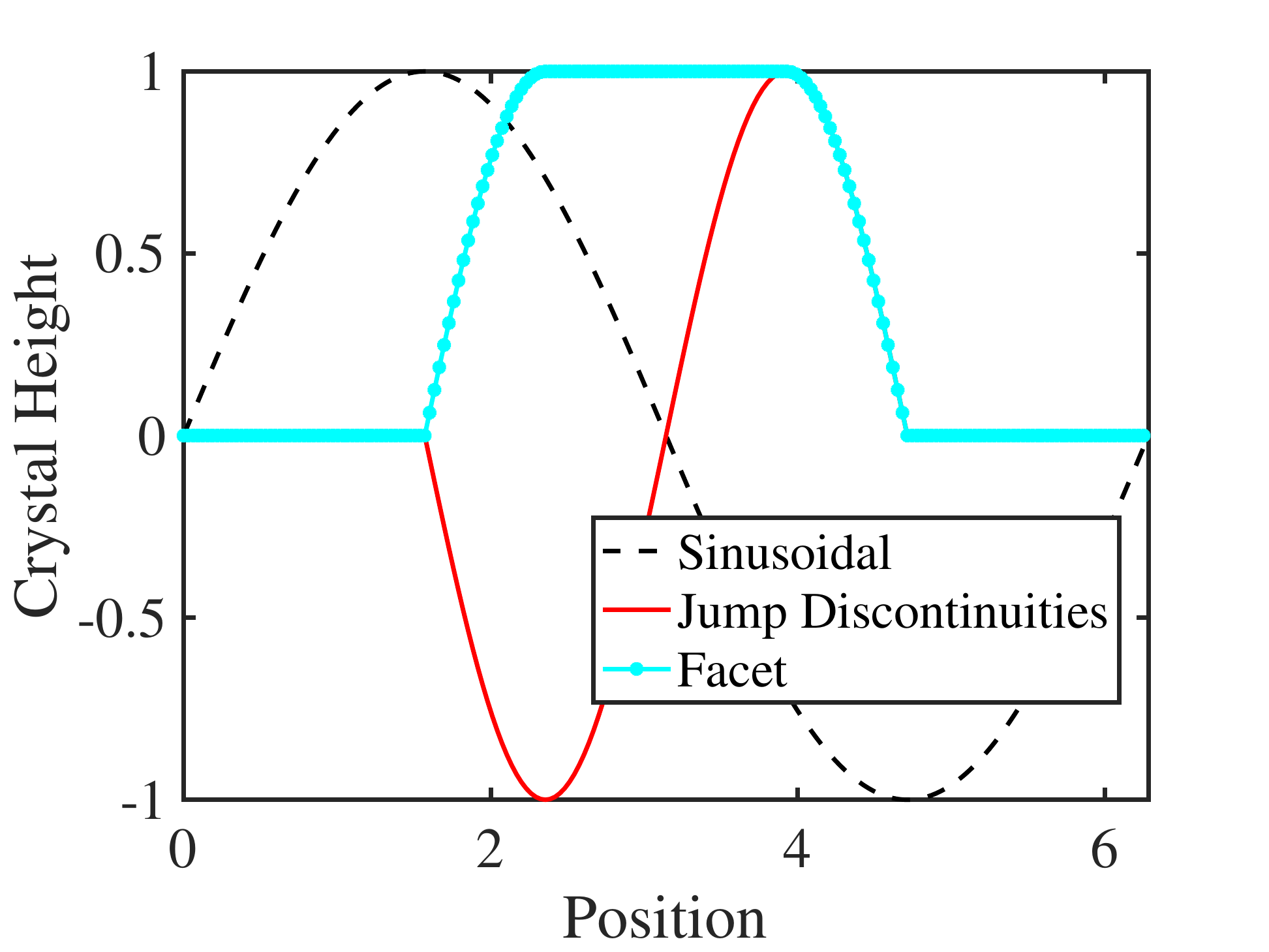}
  \begin{tabular}[b]{l}
 $h_{\text{sine}}(x) \ = \   \ \ \sin (x)$  \\
 $h_{\text{jump}}(x)  = \begin{cases} \sin (2x) & x \in (\pi/2,3 \pi/2) \\
0 & \text{otherwise}\end{cases}$  \\
$h_{\text{facet}}(x) =  \begin{cases}
\sin (2(x-\pi/2)) & x \in \left( \frac{\pi}{2}, \frac{3\pi}{4} \right) \\
1 & x \in (3\pi/4, 5\pi/4) \\
\cos (2(x-5\pi/4)) & x \in (5\pi/4, 3\pi/2) \\
0 & \text{otherwise} 
\end{cases}$ \vspace{.2cm}
  \end{tabular}
  \vspace{-2em}
\caption{Choices of initial data.}
\label{initialdata}
\end{figure}
In this section, we present a range of numerical examples illustrating the performance of the proposed algorithm.  In each test, we consider the stopping criteria  $\|( h^{(m+1)} - h^{(m)} , \phi^{(m+1)} - \phi^{(m)})\| < \delta $, where we   take the threshold $\delta  = 5 \times 10^{-6}$. Unless otherwise specified, the outer time step for the semi-implicit scheme $h^n$ is   chosen to be $\tau = T/10$, where $T$ is the final computational time, so that $N_t = 10$.  In order to ensure that the matrix inverse in the definition of $h^{(m+1)}$, equation (\ref{hmp1matrixinverse}), is well defined, we  choose $\lambda$ sufficiently large so that $\frac{\tau}{\lambda} \| \Amat \Dmat^t \Dmat \| < 1$. In the following examples, we choose $\sigma = 5 \times 10^{-4}, \lambda = 500$ for all $N_x$. We consider three choices of initial data, as shown in Figure \ref{initialdata}.

\begin{figure}[h]
\hspace{1.8cm} \footnotesize Sinuoidal \hspace{2cm} Jump Discontinuities \hspace{2.2cm}  Facet \\
\includegraphics[width=4.38cm, trim={.1cm .15cm 1.8cm .8cm},clip]{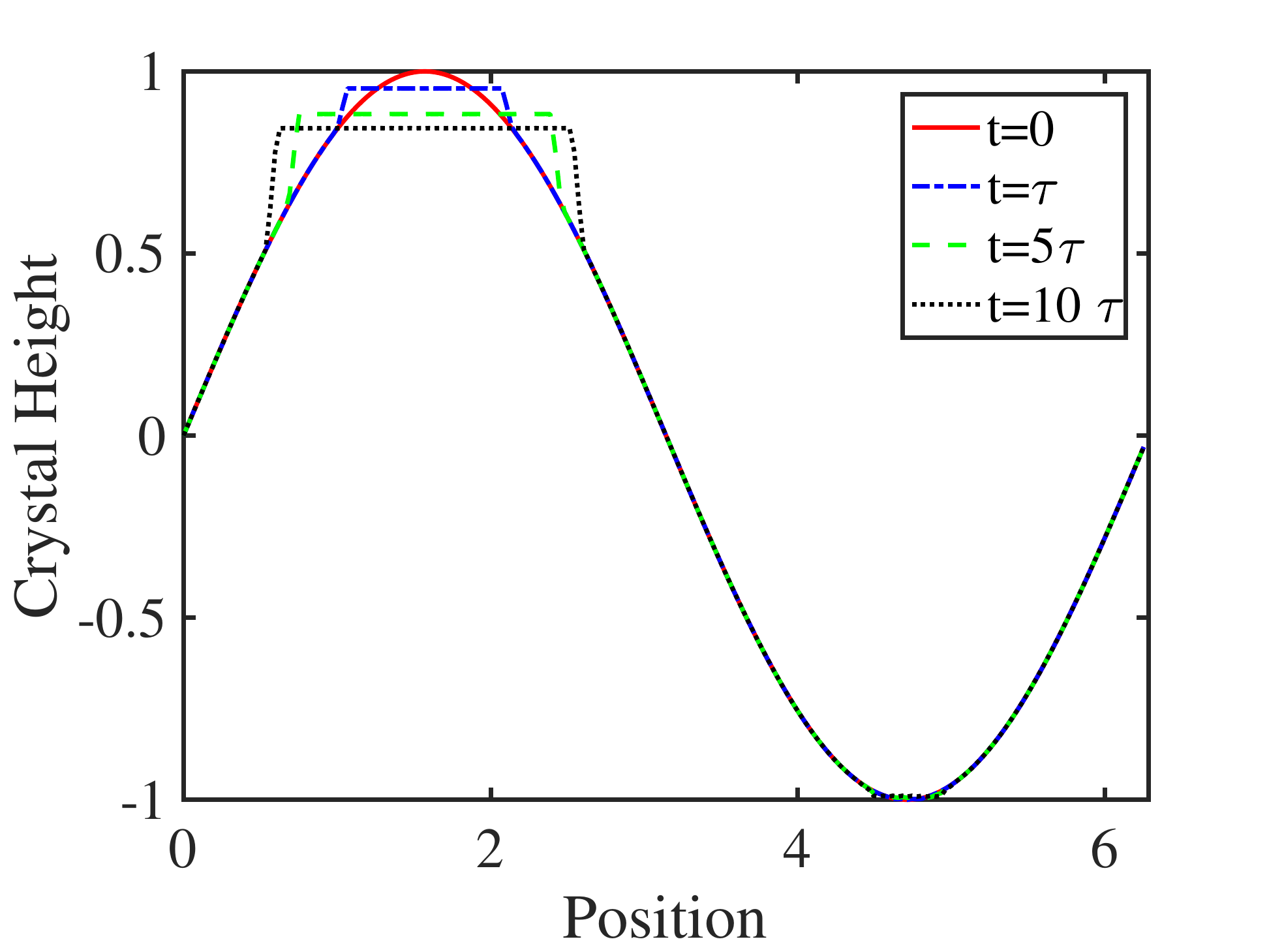} 
\includegraphics[width=4.15cm, trim={1.2cm .15cm 1.8cm .8cm},clip]{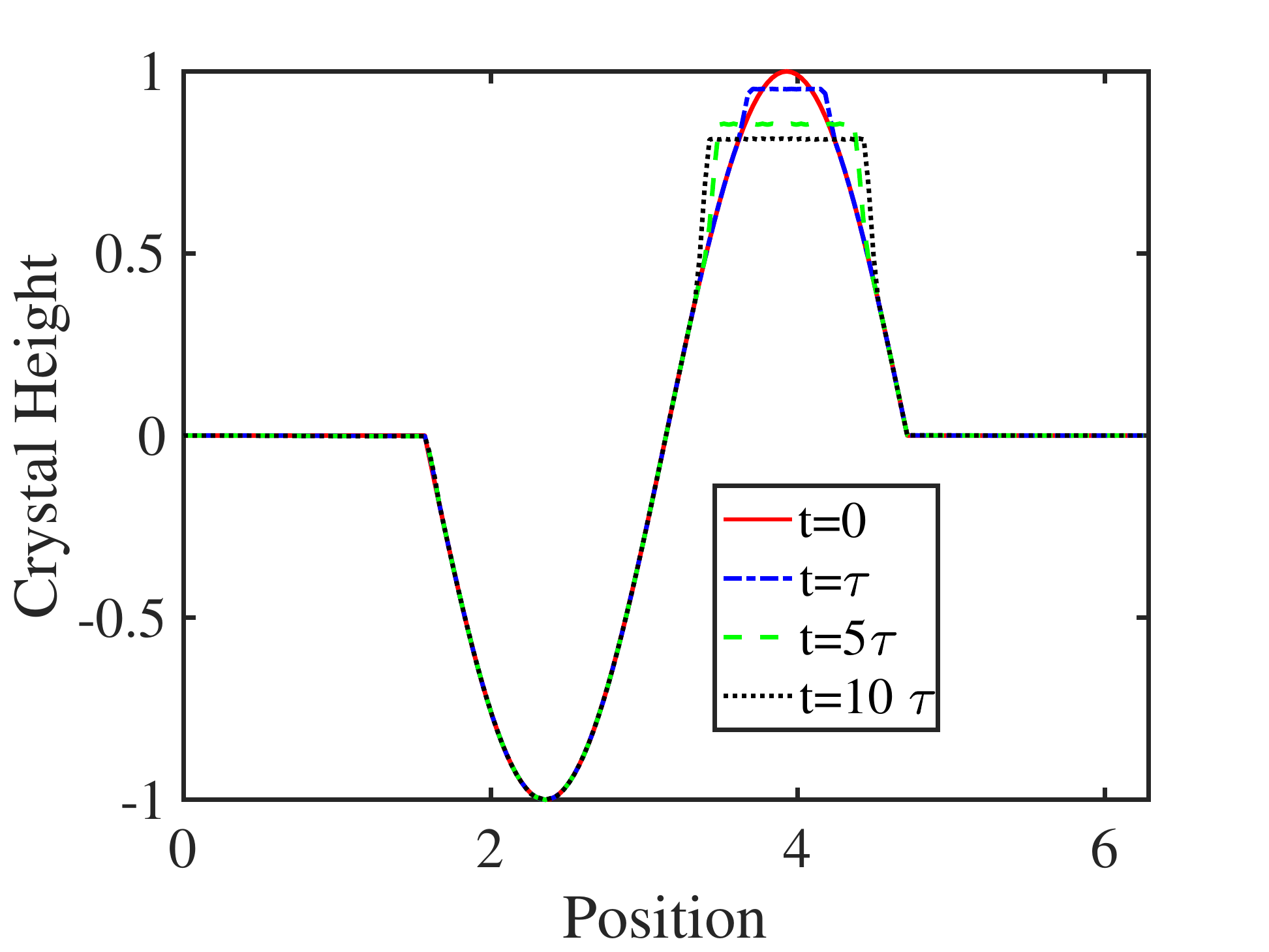} 
\includegraphics[width=4.15cm, trim={1.2cm .15cm 1.8cm .8cm},clip]{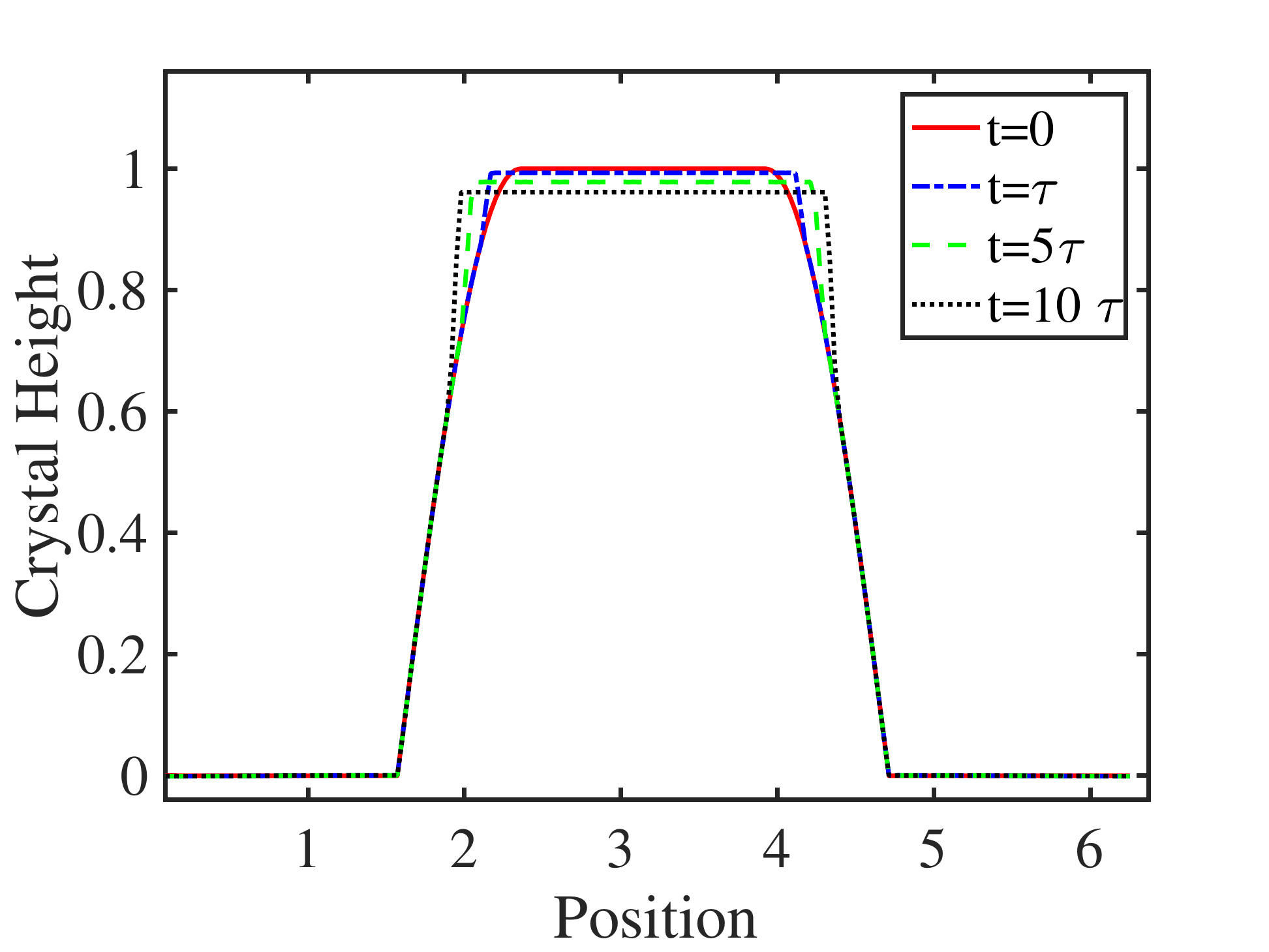} 
\caption{Dynamics of crystal surface evolution equation for different choices of initial data. Near   maxima, flat facets form and  expand outward, while   minima remain stationary.}
\label{fig:id}

\end{figure}
In Figure \ref{fig:id}, we display the dynamics of the crystal surface evolution equation for each choice of initial data. We chose $\epsilon = 0.04$, $N_x = 200$ in each of these calculations, letting $T = 10^{-2}$ in the case of the Sinusoidal and the Facet dynamics and $T = 10^{-3}$ for the Jump dynamics.  Near the maxima, flat facets expand outward like a free boundary type solution, while the minimum is   stationary, as predicted in \cite{LLMM1}.

 \begin{figure}[h]
\hspace{1.8cm} \footnotesize Sinuoidal \hspace{2cm} Jump Discontinuities \hspace{2.2cm}  Facet \\
\includegraphics[width=4.38cm, trim={.1cm 0cm 1.2cm .8cm},clip]{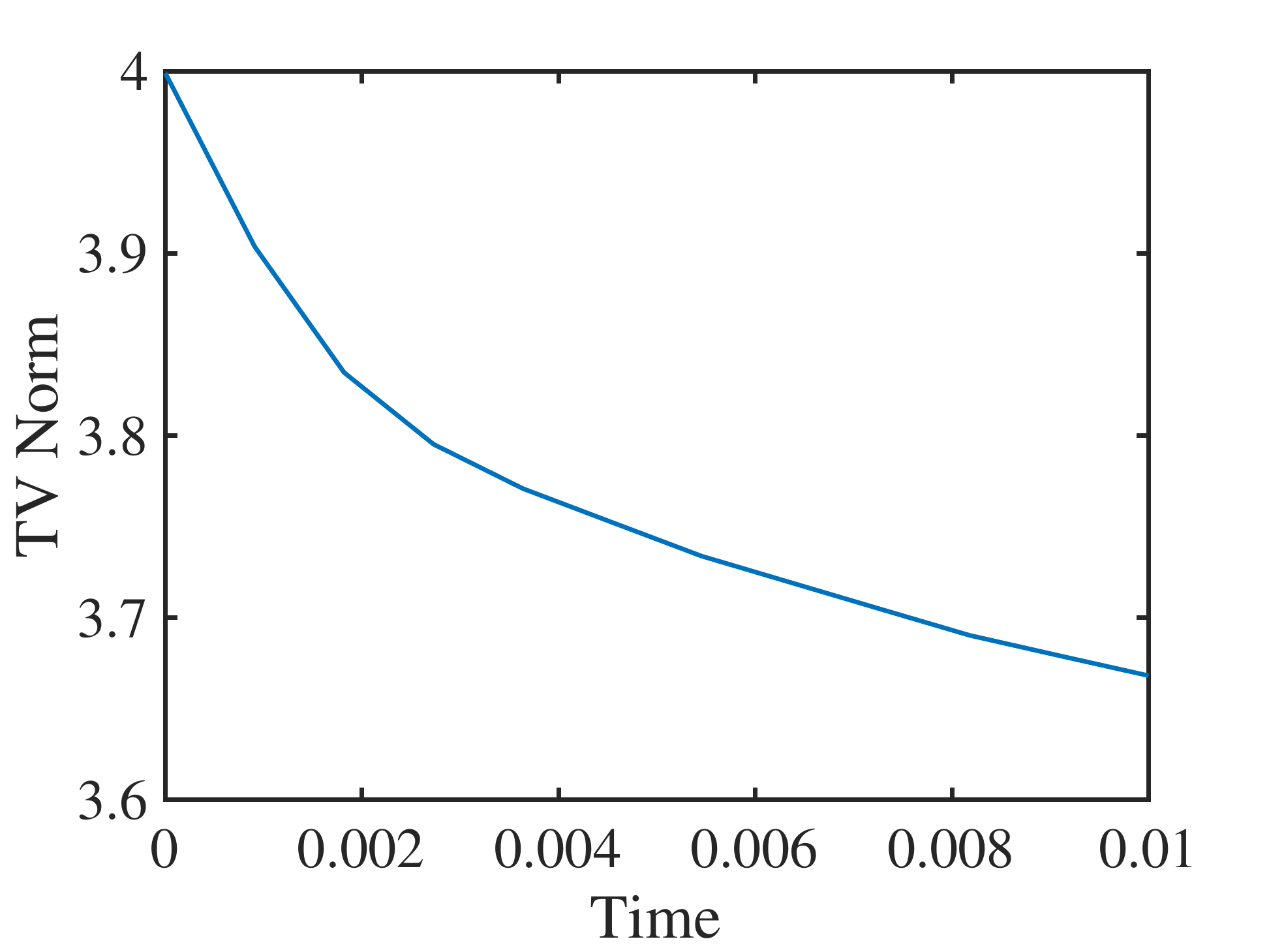} 
\includegraphics[width=4.15cm, trim={1.2cm 0cm 1.2cm .8cm},clip]{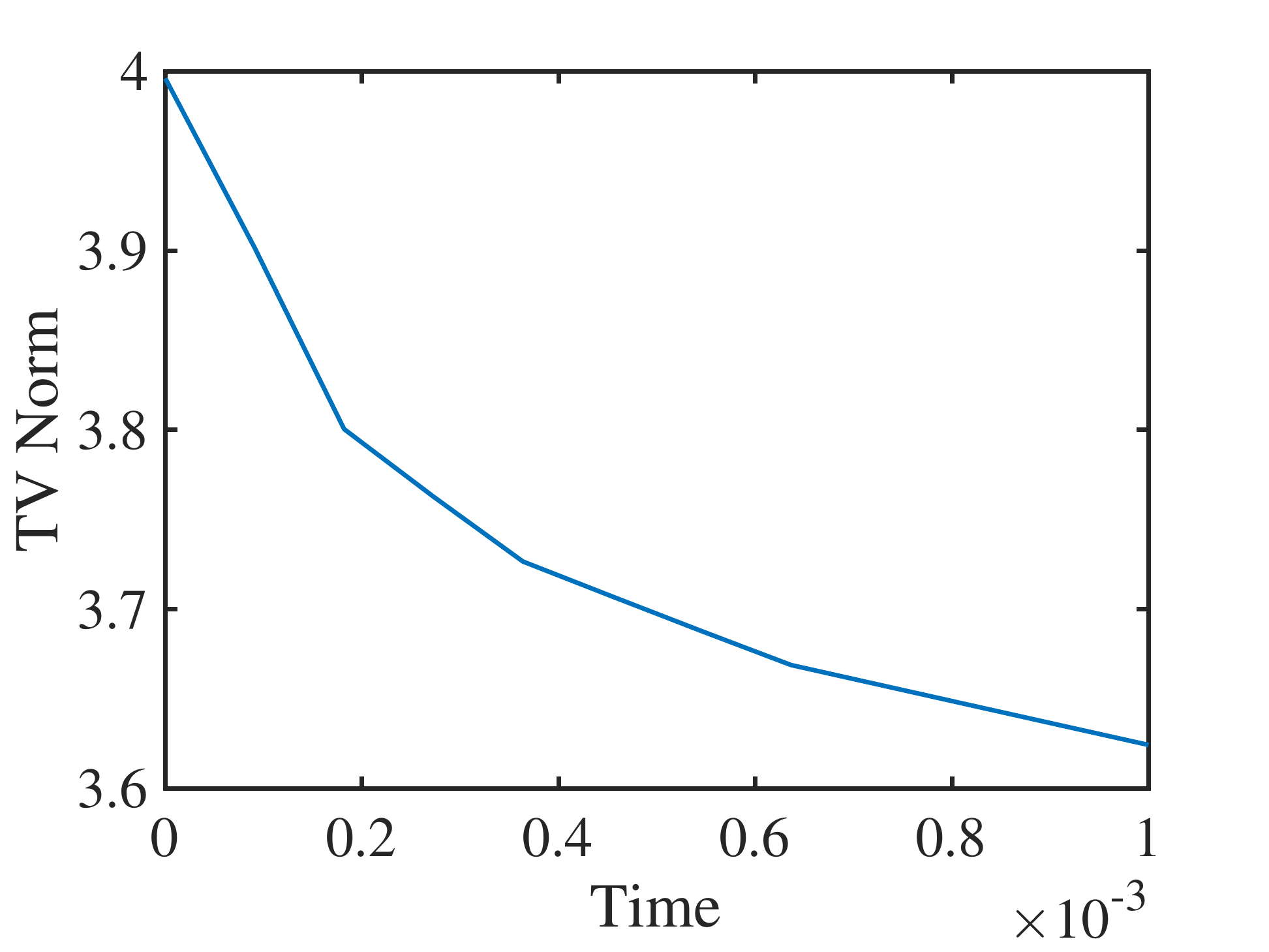} 
\includegraphics[width=4.15cm, trim={1.2cm 0cm 1.2cm .8cm},clip]{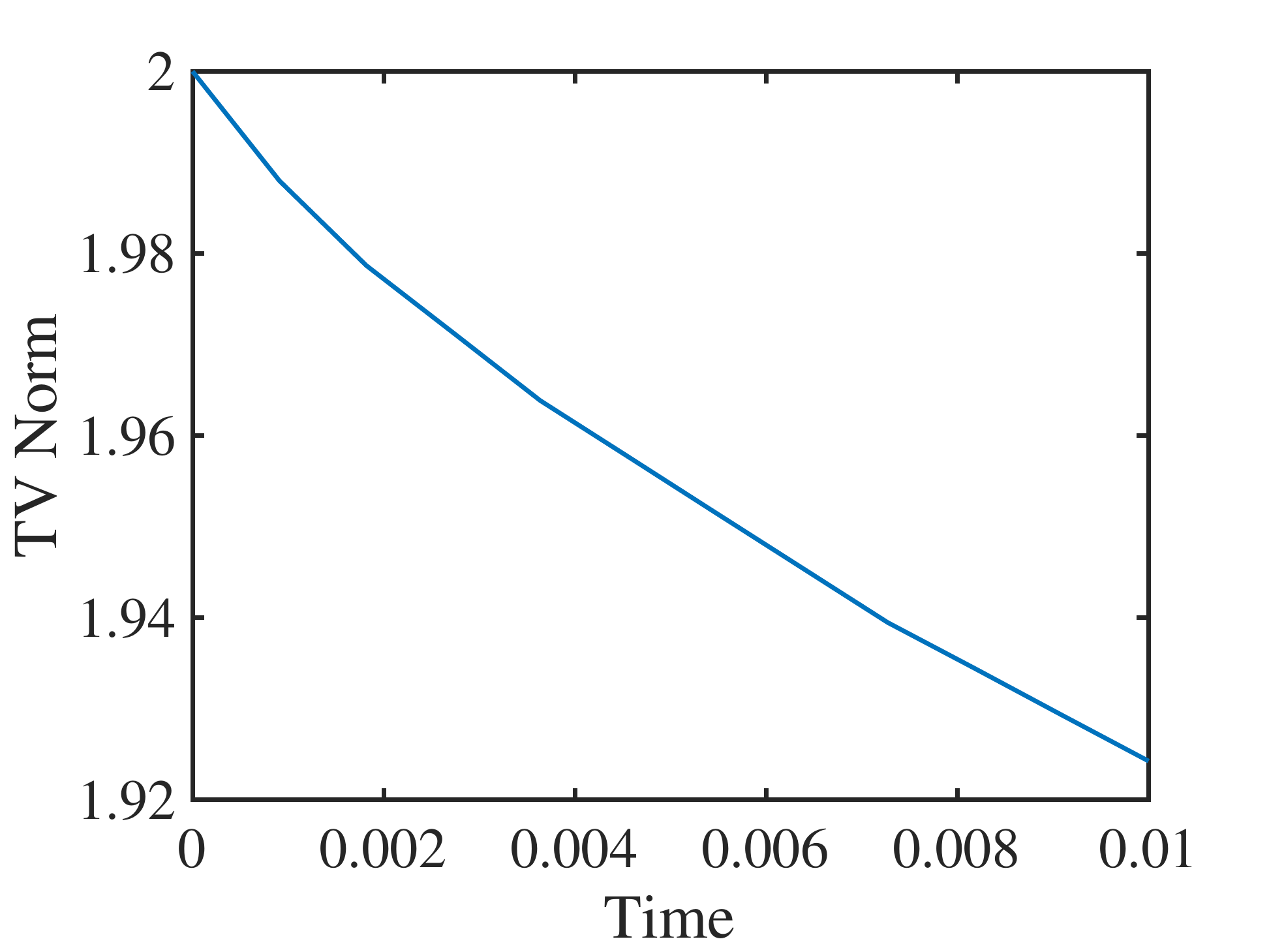} 
%\includegraphics[width=4.38cm, trim={.1cm .0cm 1.2cm 0cm},clip]{figs/sine/Sine_IntTimeSteps_alt-eps-converted-to.pdf} 
%\hspace{.03cm} \includegraphics[width=4.15cm, trim={1.8cm .0cm .6cm 0cm},clip]{figs/jump/Jump_IntTimeSteps_alt-eps-converted-to.pdf} 
% \hspace{.03cm} \includegraphics[width=4.15cm, trim={1.4cm .0cm .8cm 0cm},clip]{figs/facet/Facet_IntTimeSteps_alt-eps-converted-to.pdf} 
\includegraphics[width=4.38cm, trim={.1cm .0cm 1.2cm 0cm},clip]{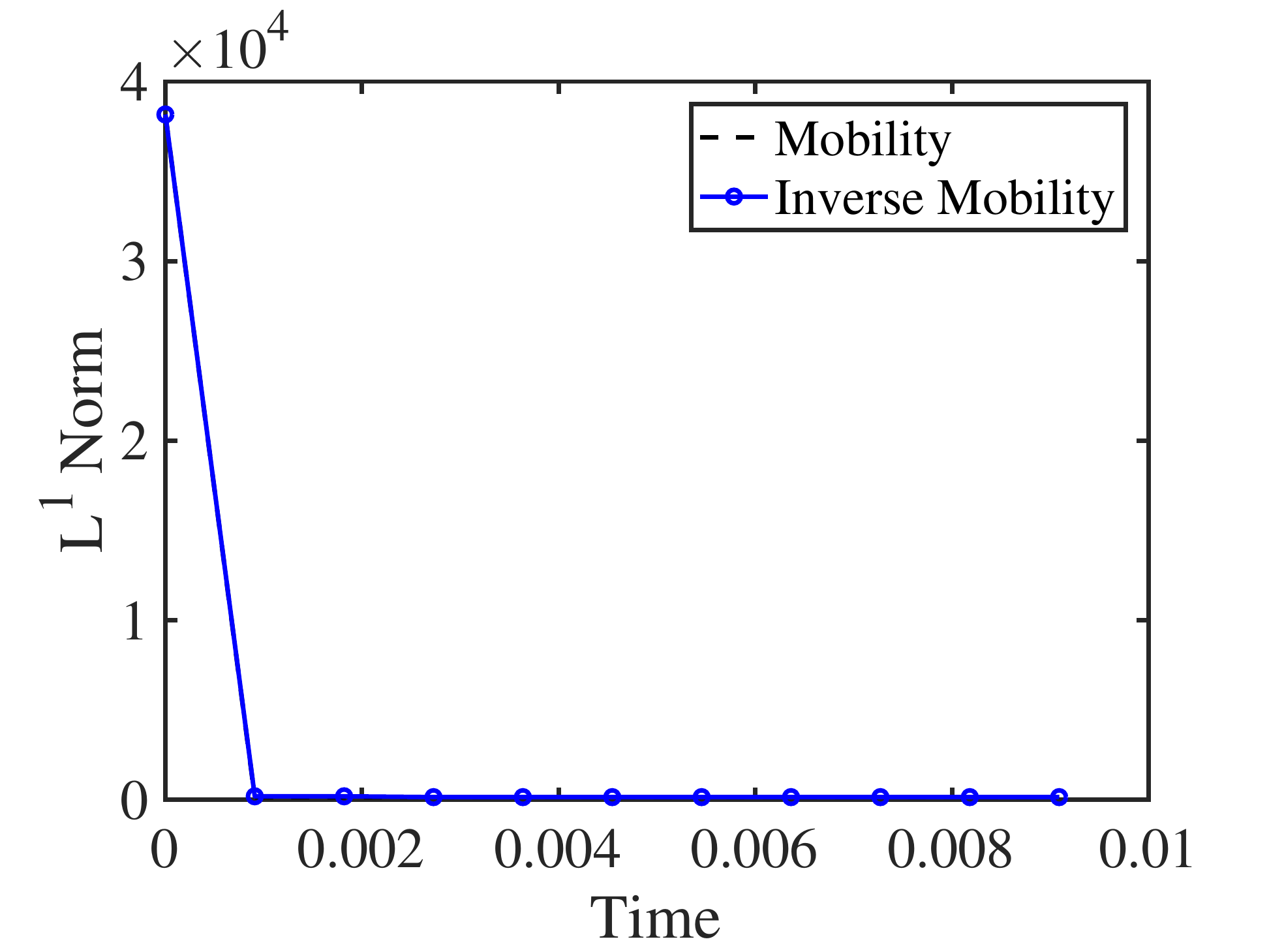} 
\hspace{.1cm} \includegraphics[width=4.18cm, trim={1.8cm .0cm .6cm 0cm},clip]{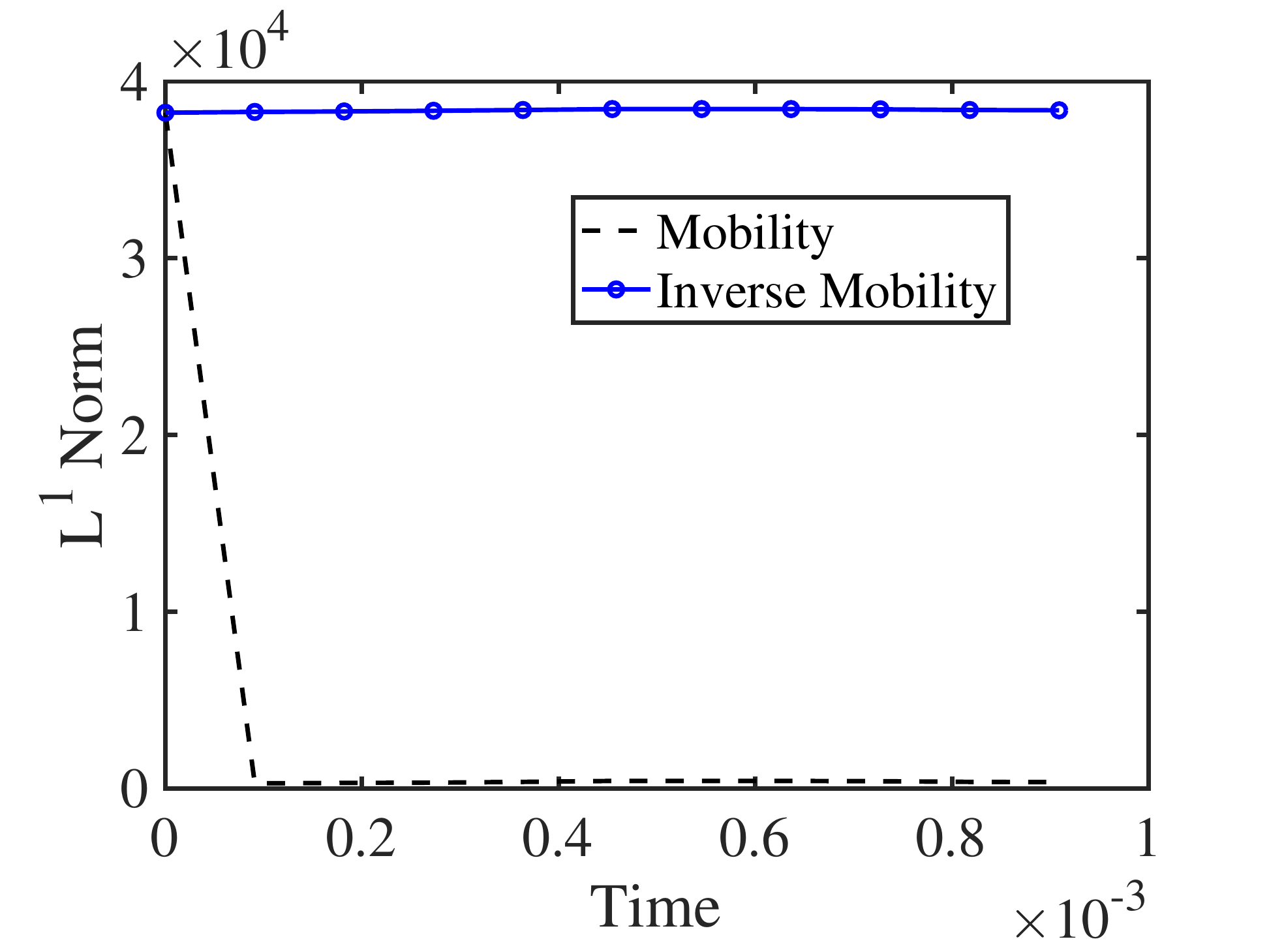} 
\includegraphics[width=4.15cm, trim={1.2cm .0cm 1.2cm 0cm},clip]{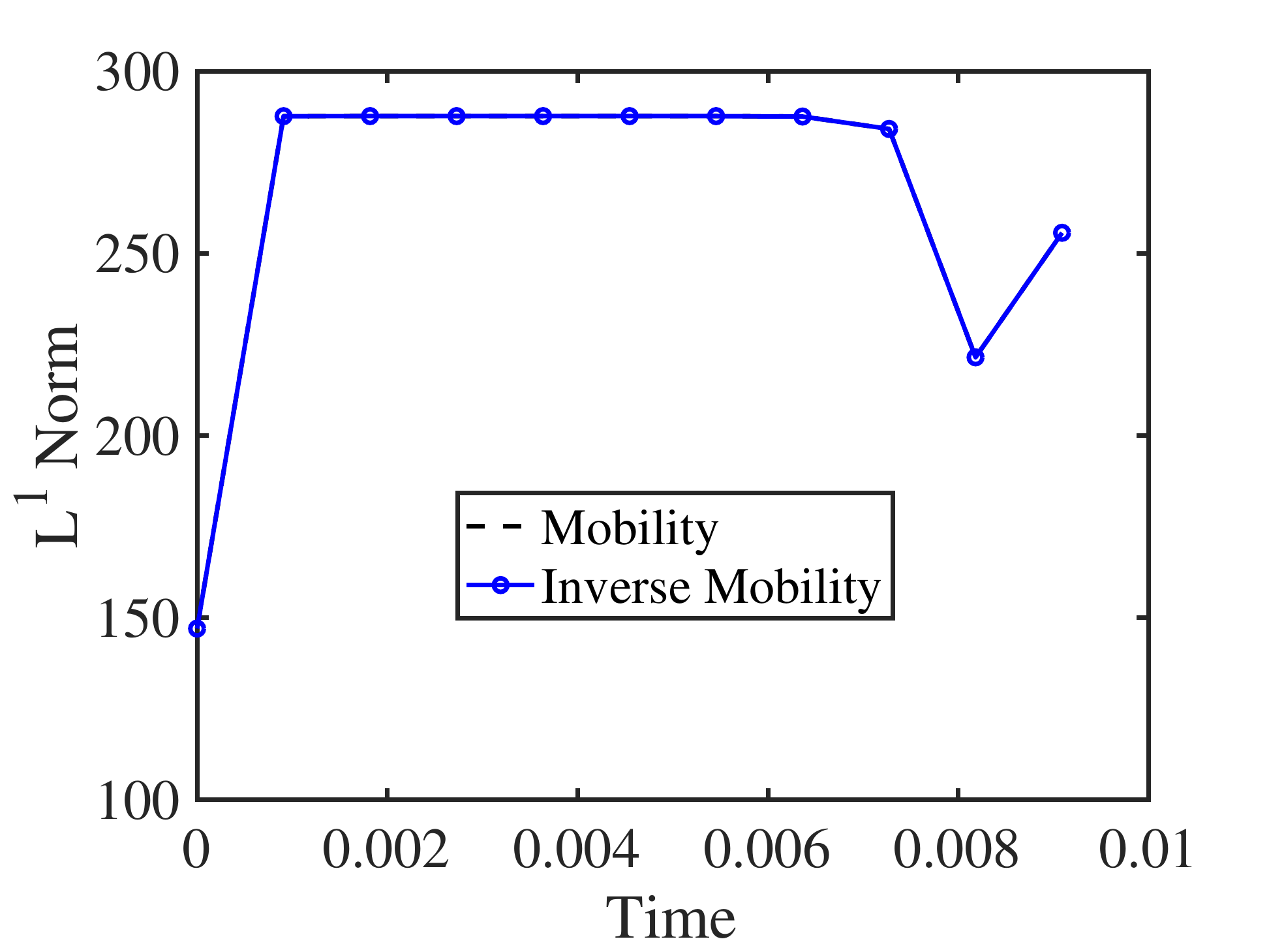} 
\caption{Top row: The total variation energy decreases in time along numerical solutions, reflecting the underlying gradient flow structure. %Middle row: The number of internal time steps required to meet convergence criteria at each external time step. 
Bottom row: The $L^1$ norms of the mobility $M(h)$ and the reciprocal of the mobility $1/M(h)$ are large, but remain bounded along the flow.}
\label{fig:combined}
\end{figure}
In Figure  \ref{fig:combined}, we analyze properties of the numerical method, under the same choices of parameters as in Figure \ref{fig:id}. In the top row, we show the decrease in the discrete TV norm $\|\Dmat h\|_1$ in time along solutions of the equation, reflecting the gradient flow structure of the equation. % In the middle row, we show the number of internal time steps of the PDHG method to converge to each step in the sequence of our outer semi-implict scheme. {\color{blue}should this be a dot plot instead of a line plot? How should we interpret these results?} 
In the bottom row, we plot the $L^1$ norms of the mobility $M(h)$ and its reciprocal $1/M(h)$. A key assumption in our convergence result for the PDHG method, Theorem \ref{mainconvresult}, is that both remain bounded, uniformly in the spatial discretization. We can see in the above simulations that, while these norms are very large, they indeed remain bounded along the flow. 

\begin{figure}[h]
\hspace{1.8cm} \footnotesize $\sgn(x) $ \hspace{2.7cm} $\tanh(10x) $ \hspace{2.8cm}  $\tanh(10x) $ \\
\includegraphics[height=3.2cm, trim={.1cm .15cm 1.8cm .8cm},clip]{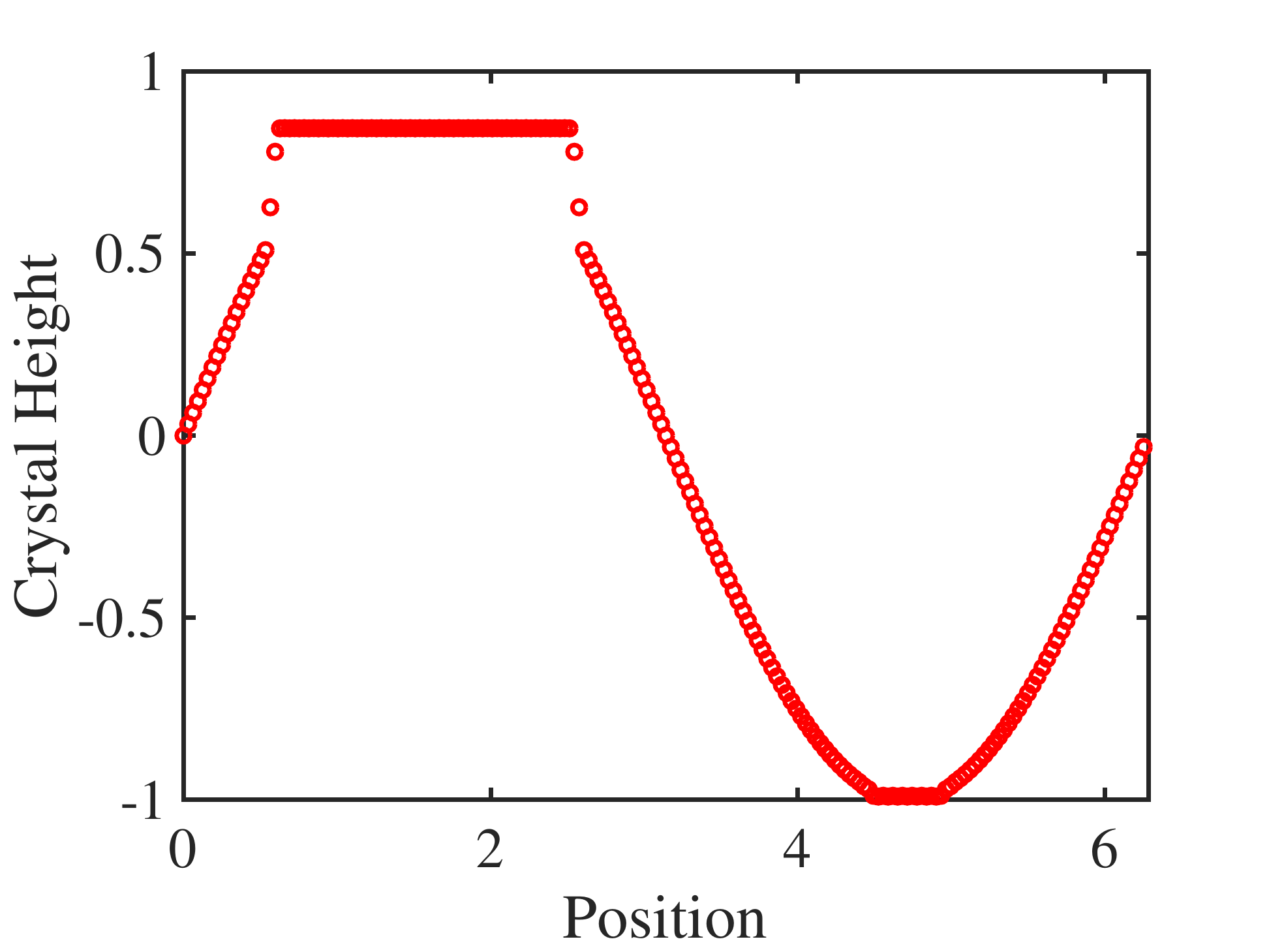} 
\includegraphics[height=3.2cm, trim={1.2cm .15cm 1.8cm .8cm},clip]{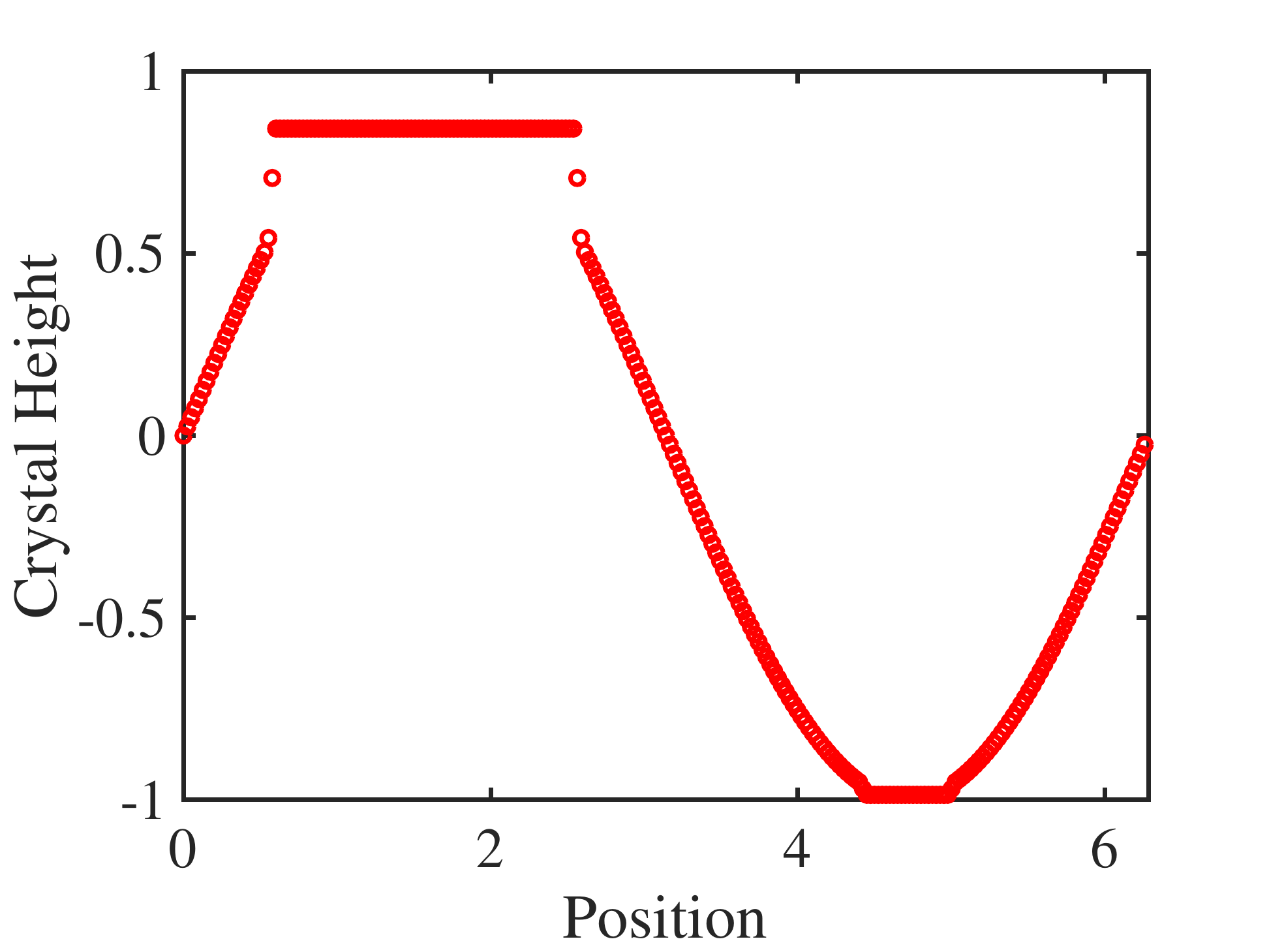}
\includegraphics[height=3.2cm, trim={.1cm .15cm 1.8cm .8cm},clip]{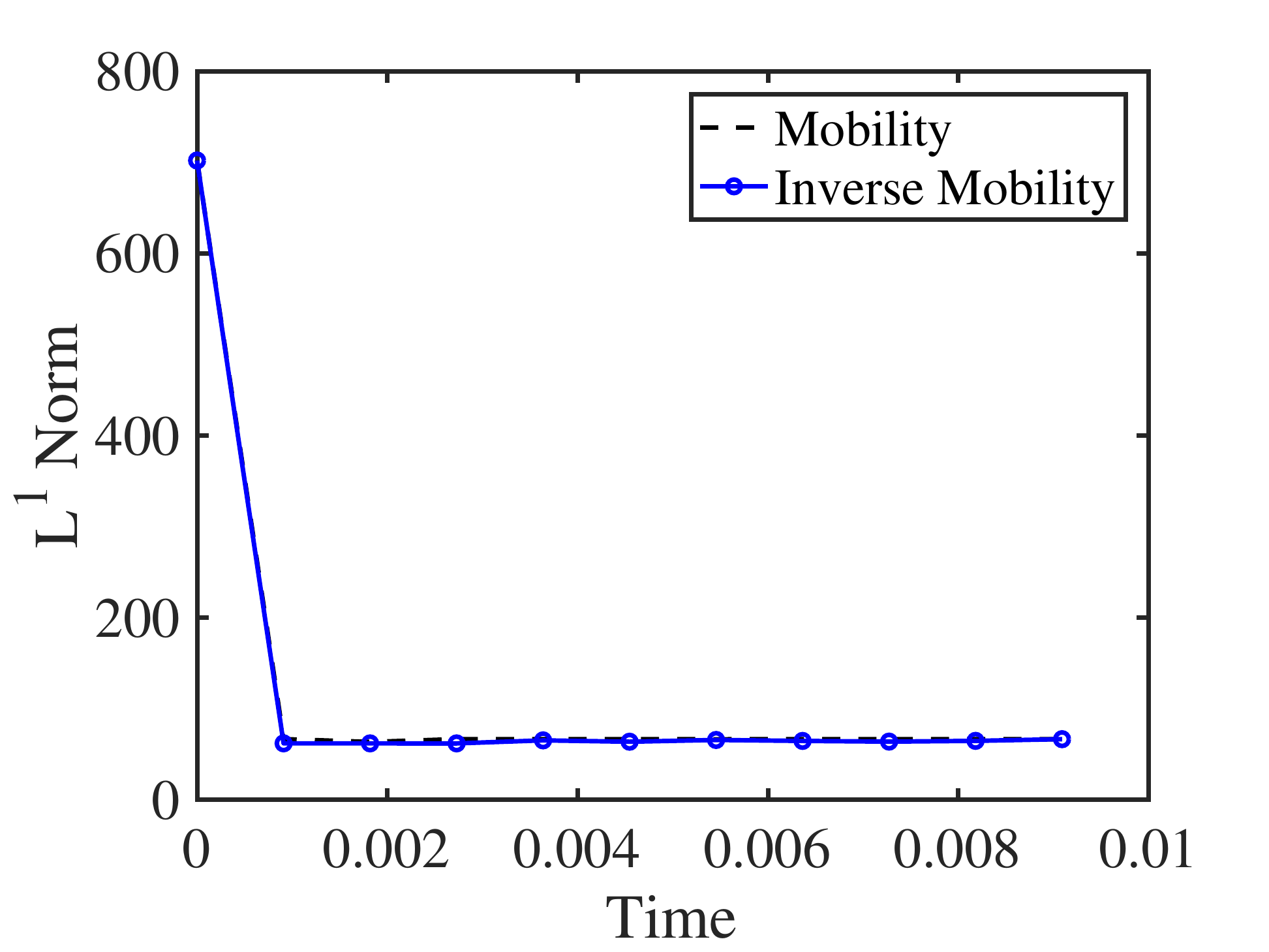}
\caption{We compare the dynamics of the mobility given by equation (\ref{MobCalc}) with a modified mobility, in which $\sgn(x)$ is replaced by $\tanh(10x)$. While the original mobility more accurately prevents facet formation at the local minimum, the modified mobility leads has smaller $L^1$ norm and requires fewer iterations to converge.}
\label{fig:mobilitycontrast}
\end{figure}

   In Figure \ref{fig:mobilitycontrast}, we compare two different choices of mobility: equation (\ref{MobCalc})  and a modified mobility, replacing $\sgn(x)$ with $\tanh (10 x)$. In both cases, we take $\epsilon = .04$. On one hand, the modified mobility has the benefit of drastically decreasing the $L^1$ norm of the mobility and its reciprocal: compare the plot on the right to the bottom left plot of Figure \ref{fig:combined}. The method also requires fewer iterations to meet the stopping criteria. On the other hand, the modified mobility allows for slightly more movement and facet formation at the minimum,  which goes against the predicted dynamics of the original equation: compare the plot on the left with the plot in the middle.

 \begin{figure}[h]
\begin{center}
\includegraphics[height=3.3cm, trim={.1cm .15cm 1.5cm .1cm},clip]{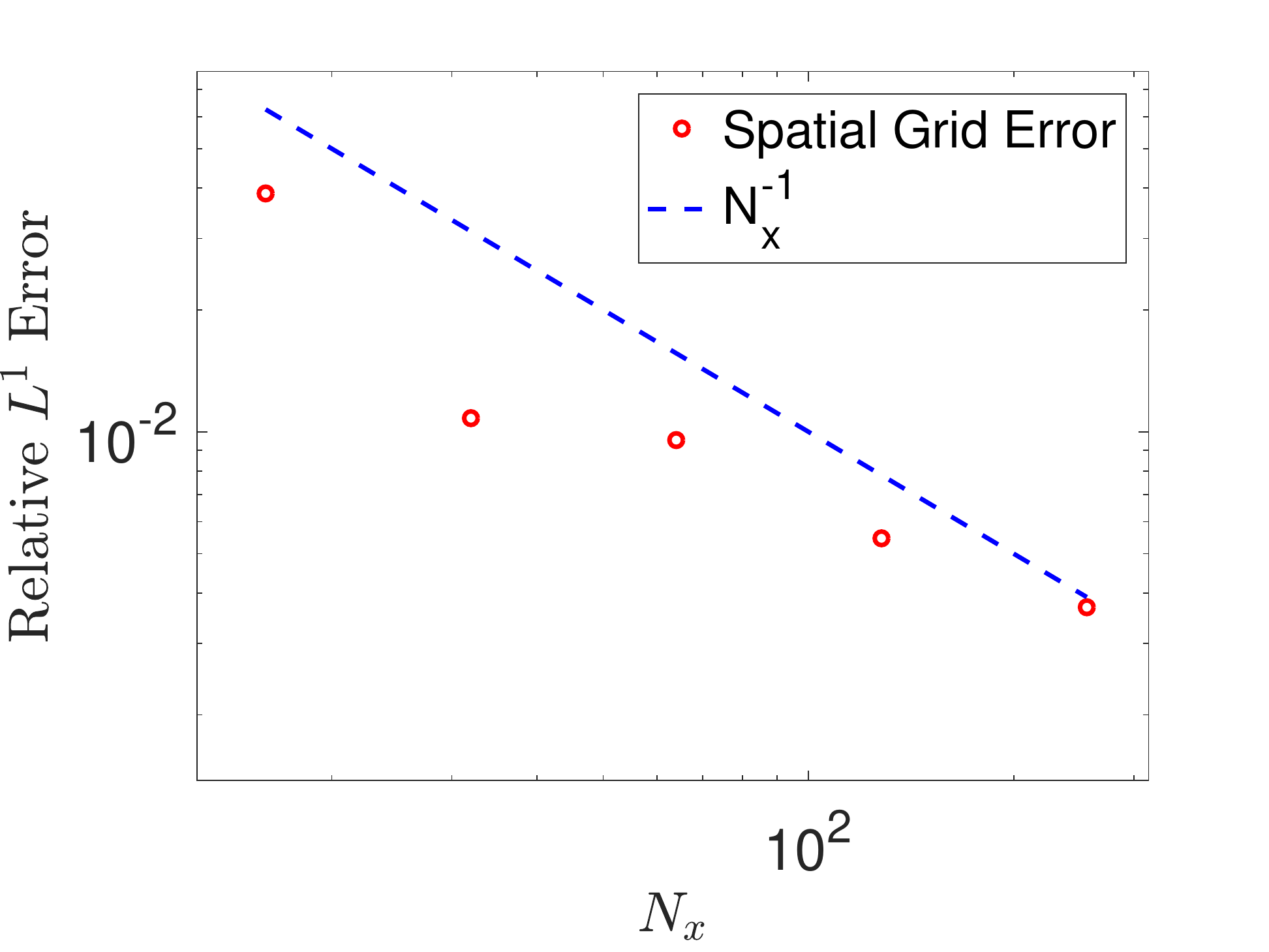}
\includegraphics[height=3.3cm, trim={.1cm .15cm 1.5cm .1cm},clip]{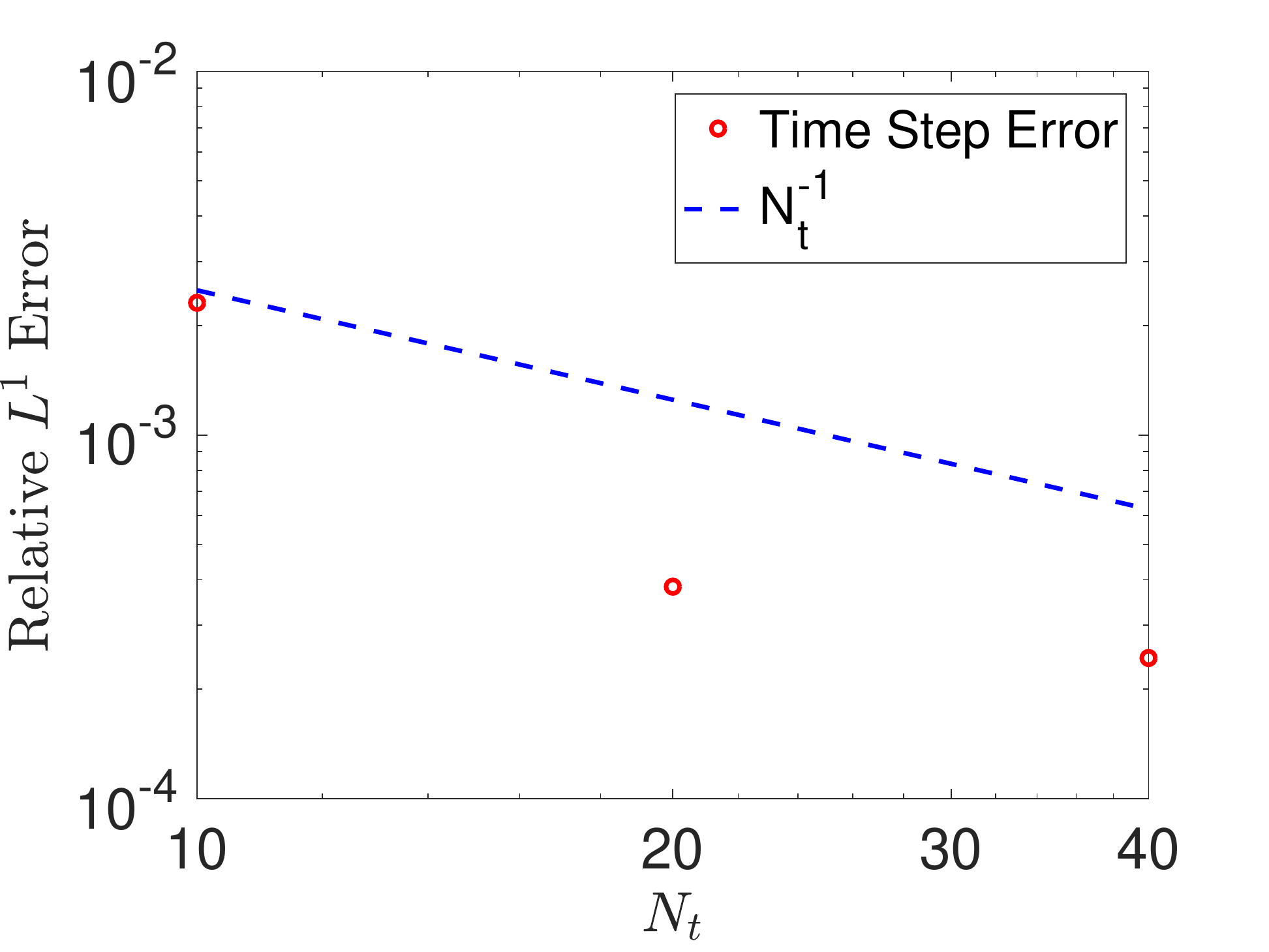}
\includegraphics[height=3.3cm, trim={.1cm .15cm 1.2cm .1cm},clip]{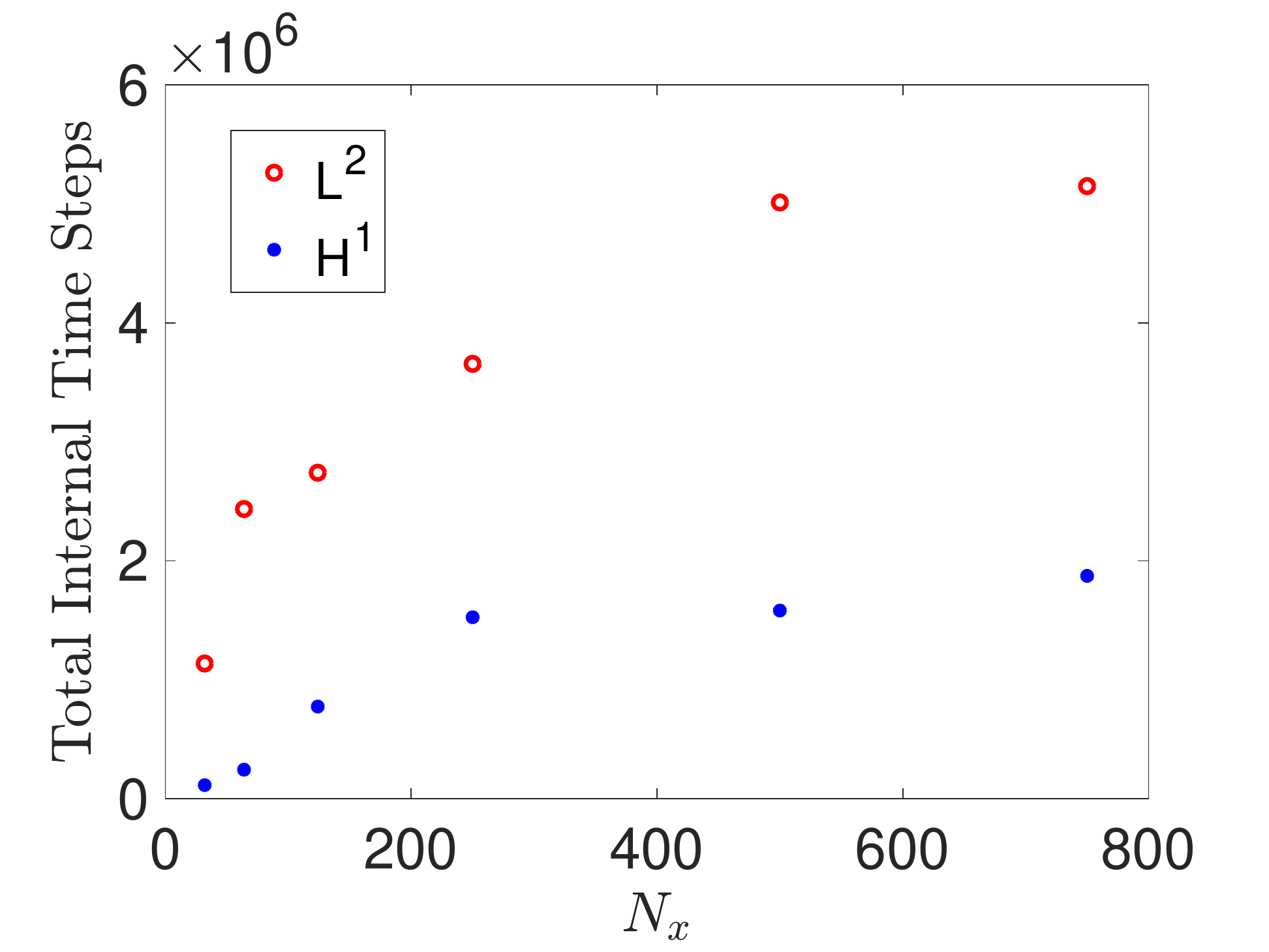}
\caption{Left: Log-Log plot of relative $L^1$ error vs. spatial grid size.  Middle:  Log-Log plot of relative $L^1$ error vs. external time step. Right:  Comparison of number of time steps required to meet stopping criteria for either $\dot{H}^1$ or $L^2$ penalization. We observe superior performance for the $\dot{H}^1$ penalization, especially as the spatial grid is refined. }
\end{center}
\label{fig:xerr}
\end{figure}

Finally, in  Figure \ref{fig:xerr}, we analyze the rate of convergence of our method.  We consider sinusoidal initial data with the modified mobility, replacing $\sgn(x)$ with $\tan(10x)$, $\epsilon = .05$ and $T = 10^{-4}$. On the left, we examine how the relative $L^1$ error  depends on the number of spatial gridpoints $N_x$ for a fixed temporal discretization, $N_t =10$. For   $N_x = 16, 32,  64,128, 256, 512$, we plot $\| h (N_x) - h(2 N_x) \|_{L^1}$. We observe slightly sublinear convergence, in line with the low spatial regularity of our solutions.

In the middle plot, we examine how the relative $L^1$ error scales with the external time step, used to define the semi-implicit scheme $h^n$ via $\tau = T/N_t$, for  a fixed spatial discretization $N_x = 256$.  For $N_t = 5,10,20,40,80$, we plot $\| h (N_t) - h(2 N_t) \|_{L^1}$.  We observe approximately first order convergence, in agreement with the interpretation of our scheme as a semi-implicit version of the minimizing movements scheme, which can be thought of as a generalized Euler method.

In the right plot, we illustrate the importance of the choice of norms in our PDHG algorithm, as explained in Remark \ref{normchoice}. At the fully discrete level, existing work \cite{
  chambolle2016ergodic} ensures that the PDHG algorithm  would converge, even if the norm penalization in the definition of $h^{(m+1)}$ was changed from a $\dot{H}^1$ norm to a $L^2$ norm. At the level of   Algorithm  \ref{alg:Hdotpen}, this would amount to modifying the computation of $h^{(m+1)}$ as follows:
 \begin{align} \label{L2hm}
h^{(m+1)} = \left(\Imat +\frac{\lambda}{\tau} \Amat^{-1}(\cdot -h^n) \right)^{-1} \left( h^{(m)}- \lambda \Dmat^t \phi^{(m)} \right) .
\end{align}
On one hand, to invert the matrix in the above formula, we need $\frac{\tau}{\lambda} \| \Amat\|< 1$. On the other hand, existing convergence results on PDHG     require $\lambda \sigma \| \Dmat^t \Dmat \| < 1$, where $\| \Dmat^t \Dmat \| \to +\infty$ as the spatial grid is refined.  These  requirements lead to significant tension  regarding the size of $\lambda$. In contrast, when choosing the  $\dot{H}^1$ norm to penalize the primal variables in our PDHG algorithm, the analogue of the constraint $\lambda \sigma \| \Dmat^t \Dmat \| < 1$ is simply $\lambda \sigma  < 1$, since the gradient is a bounded operator on $\dot{H}^1$. Thus, our method avoids this source of tension in the definition of the inner time steps $\lambda, \sigma$.

This discussion is born out numerically in the right plot above, in which we compare the number of iterations required for each method as the spatial grid is refined,  $N_x = 32,64,124,250,500,750$. We consider $T=10^{-6}$ external time steps, setting $\sigma= 5 \times 10^{-5}$, $\lambda = 5 \times 10^{-5}$ for the $L^2$ algorithm (the largest we could take to allow convergence for the $L^2$ Algorithm to still converge at all scales) and $\sigma= 5 \times 10^{-4}$, $\lambda = 500$ for our $\dot{H}^1$ algorithm,  Algorithm \ref{alg:Hdotpen}.

\bibliographystyle{abbrv}
\bibliography{KatysBib_013119}

\begin{thebibliography}{10}

\bibitem{ambrose2019radius}
D.~M. Ambrose.
\newblock The radius of analyticity for solutions to a problem in epitaxial
  growth on the torus.
\newblock {\em Bulletin of the London Mathematical Society}, 51(5):877--886,
  2019.

\bibitem{AGS}
L.~Ambrosio, N.~Gigli, and G.~Savar{\'e}.
\newblock {\em Gradient flows in metric spaces and in the space of probability
  measures}.
\newblock Lectures in Mathematics ETH Z\"urich. Birkh\"auser Verlag, Basel,
  second edition, 2008.

\bibitem{BonzelPreuss95}
H.~Bonzel and E.~Preuss.
\newblock Morphology of periodic surface profiles below the roughening
  temperature: aspects of continuum theory.
\newblock {\em Surface science}, 336(1-2):209--224, 1995.

\bibitem{BCF51}
W.-K. Burton, N.~Cabrera, and F.~Frank.
\newblock The growth of crystals and the equilibrium structure of their
  surfaces.
\newblock {\em Philosophical Transactions of the Royal Society of London.
  Series A, Mathematical and Physical Sciences}, 243(866):299--358, 1951.

\bibitem{cances2019variational}
C.~Canc{\`e}s, T.~O. Gallou{\"e}t, and G.~Todeschi.
\newblock A variational finite volume scheme for {W}asserstein gradient flows.
\newblock {\em arXiv preprint arXiv:1907.08305}, 2019.

\bibitem{carlier2017convergence}
G.~Carlier, V.~Duval, G.~Peyr{\'e}, and B.~Schmitzer.
\newblock Convergence of entropic schemes for optimal transport and gradient
  flows.
\newblock {\em SIAM Journal on Mathematical Analysis}, 49(2):1385--1418, 2017.

\bibitem{carrillo2019primal}
J.~A. Carrillo, K.~Craig, L.~Wang, and C.~Wei.
\newblock Primal dual methods for {W}asserstein gradient flows.
\newblock {\em arXiv preprint arXiv:1901.08081}, 2019.

\bibitem{carrillo2009fermi}
J.~A. Carrillo, P.~Lauren{\c{c}}ot, and J.~Rosado.
\newblock Fermi--{D}irac--{F}okker--{P}lanck equation: Well-posedness \&
  long-time asymptotics.
\newblock {\em Journal of Differential Equations}, 247(8):2209--2234, 2009.

\bibitem{chambolle2016ergodic}
A.~Chambolle and T.~Pock.
\newblock On the ergodic convergence rates of a first-order primal--dual
  algorithm.
\newblock {\em Mathematical Programming}, 159(1-2):253--287, 2016.

\bibitem{elliott1996cahn}
C.~M. Elliott and H.~Garcke.
\newblock On the {C}ahn--{H}illiard equation with degenerate mobility.
\newblock {\em Siam journal on mathematical analysis}, 27(2):404--423, 1996.

\bibitem{gao2020analysis}
Y.~Gao, A.~E. Katsevich, J.-G. Liu, J.~Lu, and J.~L. Marzuola.
\newblock Analysis of a fourth order exponential pde arising from a crystal
  surface jump process with {M}etropolis-type transition rates.
\newblock {\em arXiv preprint arXiv:2003.07236}, 2020.

\bibitem{gao2019analysis}
Y.~Gao, J.-G. Liu, J.~Lu, and J.~L. Marzuola.
\newblock Analysis of a continuum theory for broken bond crystal surface models
  with evaporation and deposition effects.
\newblock {\em Nonlinearity}, 33:3816–3845, 2020.

\bibitem{GigaGiga}
M.-H. Giga and Y.~Giga.
\newblock Very singular diffusion equations: second and fourth order problems.
\newblock {\em Japan journal of industrial and applied mathematics},
  27(3):323--345, 2010.

\bibitem{GigaKohn}
Y.~Giga and R.~V. Kohn.
\newblock Scale-invariant extinction time estimates for some singular diffusion
  equations.
\newblock {\em Discrete Contin. Dyn. Syst}, 30(2):509--535, 2011.

\bibitem{Giga1}
Y.~Giga, H.~Kuroda, and H.~Matsuoka.
\newblock Fourth-order total variation flow with {D}irichlet condition:
  characterization of evolution and extinction time estimates.
\newblock {\em Hokkaido University Preprint Series in Mathematics}, 1064:1--36,
  2015.

\bibitem{granero2018global}
R.~Granero-Belinch{\'o}n and M.~Magliocca.
\newblock Global existence and decay to equilibrium for some crystal surface
  models.
\newblock {\em arXiv preprint arXiv:1804.09645}, 2018.

\bibitem{GruberMullins67}
E.~Gruber and W.~Mullins.
\newblock On the theory of anisotropy of crystalline surface tension.
\newblock {\em Journal of Physics and Chemistry of Solids}, 28(5):875--887,
  1967.

\bibitem{he2012convergence}
B.~He and X.~Yuan.
\newblock Convergence analysis of primal-dual algorithms for a saddle-point
  problem: from contraction perspective.
\newblock {\em SIAM Journal on Imaging Sciences}, 5(1):119--149, 2012.

\bibitem{IhleMisbahP-L98}
T.~Ihle, C.~Misbah, and O.~Pierre-Louis.
\newblock Equilibrium step dynamics on vicinal surfaces revisited.
\newblock {\em Physical Review B}, 58(4):2289, 1998.

\bibitem{jacobs2019solving}
M.~Jacobs, F.~L{\'e}ger, W.~Li, and S.~Osher.
\newblock Solving large-scale optimization problems with a convergence rate
  independent of grid size.
\newblock {\em SIAM Journal on Numerical Analysis}, 57(3):1100--1123, 2019.

\bibitem{JKO}
R.~Jordan, D.~Kinderlehrer, and F.~Otto.
\newblock The variational formulation of the {F}okker-{P}lanck equation.
\newblock {\em SIAM J. Math. Anal.}, 29(1):1--17, 1998.

\bibitem{KobayashiGiga99}
R.~Kobayashi and Y.~Giga.
\newblock Equations with singular diffusivity.
\newblock {\em Journal of statistical physics}, 95(5-6):1187--1220, 1999.

\bibitem{KohnV}
R.~V. Kohn and H.~M. Versieux.
\newblock Numerical analysis of a steepest-descent pde model for surface
  relaxation below the roughening temperature.
\newblock {\em SIAM journal on numerical analysis}, 48(5):1781--1800, 2010.

\bibitem{Sethna96}
B.~Krishnamachari, J.~McLean, B.~Cooper, and J.~Sethna.
\newblock Gibbs-{T}homson formula for small island sizes: Corrections for high
  vapor densities.
\newblock {\em Physical Review B}, 54(12):8899, 1996.

\bibitem{KDM}
J.~Krug, H.~Dobbs, and S.~Majaniemi.
\newblock Adatom mobility for the solid-on-solid model.
\newblock {\em Zeitschrift f{\"u}r Physik B Condensed Matter}, 97(2):281--291,
  1995.

\bibitem{li2020fisher}
W.~Li, J.~Lu, and L.~Wang.
\newblock Fisher information regularization schemes for {W}asserstein gradient
  flows.
\newblock {\em Journal of Computational Physics}, page 109449, 2020.

\bibitem{liero2013gradient}
M.~Liero and A.~Mielke.
\newblock Gradient structures and geodesic convexity for reaction--diffusion
  systems.
\newblock {\em Philosophical Transactions of the Royal Society A: Mathematical,
  Physical and Engineering Sciences}, 371(2005):20120346, 2013.

\bibitem{lisini2012cahn}
S.~Lisini, D.~Matthes, and G.~Savar{\'e}.
\newblock Cahn--{H}illiard and thin film equations with nonlinear mobility as
  gradient flows in weighted-{W}asserstein metrics.
\newblock {\em Journal of Differential Equations}, 253(2):814--850, 2012.

\bibitem{LLMM1}
J.-G. Liu, J.~Lu, D.~Margetis, and J.~L. Marzuola.
\newblock Asymmetry in crystal facet dynamics of homoepitaxy by a continuum
  model.
\newblock {\em Physica D: Nonlinear Phenomena}, 393:54--67, 2019.

\bibitem{liu2018global}
J.-G. Liu and R.~M. Strain.
\newblock Global stability for solutions to the exponential pde describing
  epitaxial growth.
\newblock {\em Interfaces and Free Boundaries}, 21:51--86, 2019.

\bibitem{liu2016existence}
J.-G. Liu and X.~Xu.
\newblock Existence theorems for a multidimensional crystal surface model.
\newblock {\em SIAM Journal on Mathematical Analysis}, 48(6):3667--3687, 2016.

\bibitem{liu2017analytical}
J.-G. Liu and X.~Xu.
\newblock Analytical validation of a continuum model for the evolution of a
  crystal surface in multiple space dimensions.
\newblock {\em SIAM Journal on Mathematical Analysis}, 49(3):2220--2245, 2017.

\bibitem{DM_Kohn06}
D.~Margetis and R.~V. Kohn.
\newblock Continuum relaxation of interacting steps on crystal surfaces in 2+1
  dimensions.
\newblock {\em Multiscale Modeling \& Simulation}, 5(3):729--758, 2006.

\bibitem{MW1}
J.~L. Marzuola and J.~Weare.
\newblock Relaxation of a family of broken-bond crystal-surface models.
\newblock {\em Physical Review E}, 88(3):032403, 2013.

\bibitem{Srolovitz94}
R.~Najafabadi and D.~J. Srolovitz.
\newblock Elastic step interactions on vicinal surfaces of fcc metals.
\newblock 1994.

\bibitem{Nessyahu:Tadmor}
H.~Nessyahu and E.~Tadmor.
\newblock Non-oscillatory central differencing for hyperbolic conservation
  laws.
\newblock {\em J. Comput. Phys.}, 87:408--463, 1990.

\bibitem{Odisharia_06}
I.~V. Odisharia.
\newblock {\em Simulation and analysis of the relaxation of a crystalline
  surface}.
\newblock New York University, 2006.

\bibitem{O}
F.~Otto.
\newblock The geometry of dissipative evolution equations: the porous medium
  equation.
\newblock {\em Comm. Partial Differential Equations}, 26(1-2):101--174, 2001.

\bibitem{Rowlinson}
J.~S. Rowlinson and B.~Widom.
\newblock {\em Molecular theory of capillarity}.
\newblock Courier Corporation, 2013.

\bibitem{rudin1992nonlinear}
L.~I. Rudin, S.~Osher, and E.~Fatemi.
\newblock Nonlinear total variation based noise removal algorithms.
\newblock {\em Physica D: nonlinear phenomena}, 60(1-4):259--268, 1992.

\bibitem{shefi2014rate}
R.~Shefi and M.~Teboulle.
\newblock Rate of convergence analysis of decomposition methods based on the
  proximal method of multipliers for convex minimization.
\newblock {\em SIAM Journal on Optimization}, 24(1):269--297, 2014.

\bibitem{Shenoy02}
V.~Shenoy and L.~Freund.
\newblock A continuum description of the energetics and evolution of stepped
  surfaces in strained nanostructures.
\newblock {\em Journal of the Mechanics and Physics of Solids},
  50(9):1817--1841, 2002.

\bibitem{wang2011error}
J.~Wang and B.~J. Lucier.
\newblock Error bounds for finite-difference methods for
  {R}udin--{O}sher--{F}atemi image smoothing.
\newblock {\em SIAM Journal on Numerical Analysis}, 49(2):845--868, 2011.

\bibitem{Zangwill91}
A.~Zangwill, C.~Luse, D.~Vvedensky, and M.~Wilby.
\newblock Equations of motion for epitaxial growth.
\newblock {\em Surface science}, 274(2):L529--L534, 1992.

\bibitem{ziemer2012weakly}
W.~P. Ziemer.
\newblock {\em Weakly differentiable functions: Sobolev spaces and functions of
  bounded variation}, volume 120.
\newblock Springer Science \& Business Media, 2012.

\end{thebibliography}

\end{document}